\documentclass[final,leqno]{siamltex}
\usepackage{amsmath}
\usepackage{amssymb}
\usepackage{array}
\usepackage{array}
\usepackage{graphicx}
\usepackage{tikz,ulem}
\usepackage{dsfont}
\usepackage{subfigure,url}

\usepackage{algpseudocode,algorithm}

\newcommand{\K}{\mathcal{K}}
\newcommand{\R}{\mathbb{R}}

\newcommand{\mt}{\color{red}}

\newtheorem{remark}{Remark}

\title{Rosenbrock-Krylov Methods for Large Systems of Differential Equations}

\author{Paul Tranquilli \footnotemark[2] \footnotemark[4] 
\and Adrian Sandu \footnotemark[3] \footnotemark[4]}

\begin{document}
\thispagestyle{empty}
\setcounter{page}{0}

\begin{Huge}
\begin{center}
%Computer Science Technical Report CSTR-{\tt insert number here} \\
Computational Science Laboratory Technical Report CSL-TR-1-2013\\
\today
\end{center}
\end{Huge}
\vfil
\begin{huge}
\begin{center}
Paul Tranquilli and Adrian Sandu
\end{center}
\end{huge}

\vfil
\begin{huge}
\begin{it}
\begin{center}
``Rosenbrock-Krylov Methods for Large Systems of Differential Equations''
\end{center}
\end{it}
\end{huge}
\vfil
\textbf{Cite as:} Paul Tranquilli and Adrian Sandu.  Rosenbrock-Krylov methods for large systems of differential equations. SIAM Journal of Scientific Computing. Volume 36, Issue 3, Pages A1313 -- A1338, 2014.
\vfil

\begin{large}
\begin{center}
Computational Science Laboratory \\
Computer Science Department \\
Virginia Polytechnic Institute and State University \\
Blacksburg, VA 24060 \\
Phone: (540)-231-2193 \\
Fax: (540)-231-6075 \\ 
Email: \url{sandu@cs.vt.edu} \\
Web: \url{http://csl.cs.vt.edu}
\end{center}
\end{large}

\vspace*{1cm}

\begin{tabular}{ccc}
\includegraphics[width=2.5in]{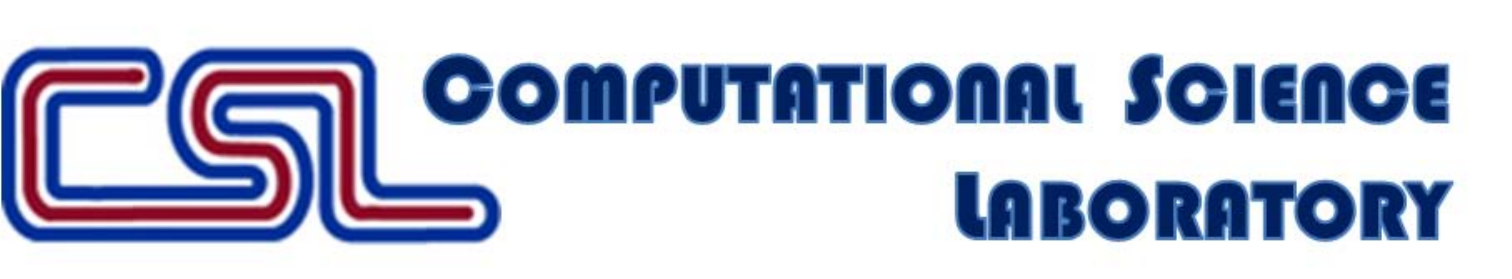}
&\hspace{2.5in}&
\includegraphics[width=2.5in]{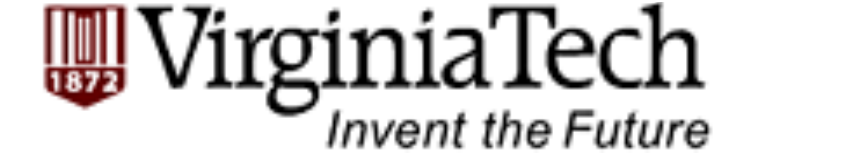} \\
{\bf\em Innovative Computational Solutions} &&\\
\end{tabular}

\newpage

\maketitle

%\begin{frontmatter}
%  \author[csl]{Paul Tranquilli}
%  \ead{ptranq@vt.edu}
%  \address[csl]{Computational Science Laboratory, Department of Computer Science, Virginia Tech.  2202 Kraft Drive, Blacksburg, Virginia 24060.}
%  \author[csl]{Adrian Sandu}
%  \ead{sandu@cs.vt.edu}
%\end{frontmatter}

\renewcommand{\thefootnote}{\fnsymbol{footnote}}
\footnotetext[2]{E-mail: ptranq@vt.edu}
\footnotetext[3]{E-mail: sandu@cs.vt.edu}
\footnotetext[4]{Computational Science Laboratory, Department of Computer Science, Virginia Tech.  2202 Kraft Drive, Blacksburg, Virginia 24060.}
\renewcommand{\thefootnote}{\arabic{footnote}}

\begin{abstract}
   This paper develops a new class of Rosenbrock-type integrators based on a Krylov space solution of the linear systems. The new family, called Rosenbrock-Krylov (Rosenbrock-K), is well suited for solving large-scale systems of ODEs or semi-discrete PDEs.  The time discretization and the Krylov space approximation are treated as a single computational process, and the
Krylov space properties are an integral part of the new Rosenbrock-K order condition theory developed herein. Consequently, Rosenbrock-K methods require a small number of basis vectors determined solely by the temporal order of accuracy. The subspace size is independent of the ODE under consideration, and there is no need to monitor the errors in linear system solutions at each stage. Numerical results show favorable properties of Rosenbrock-K methods when compared to current Rosenbrock and Rosenbrock-W schemes. 
\end{abstract}
\pagestyle{myheadings}
\thispagestyle{plain}

%%%%%%%%%%%%%%%%%%%%%%%%
\section{Introduction}
%%%%%%%%%%%%%%%%%%%%%%%%
This paper is concerned with the numerical solution of large initial value problems of the form
\begin{equation}
\label{eqn:ode}
 \frac{dy}{dt} = f(t,y)\,,~~~ t_0 \leq t \leq t_F\,, ~~~ y(t_0) = y_0\,; \quad  y(t), f(t,y) \in \R^N\,.
\end{equation}
Ordinary differential equations (ODEs) model the evolution of chemical kinetic systems, electronic circuits, or the motion of planets.
Other physical systems, such as dynamics of the atmosphere and oceans, or the behavior of materials, require the solution of a system of partial differential equations (PDEs). A standard approach for solving evolutionary PDEs is to discretize in space, reducing the problem to a large system of ODEs which are then integrated forward in time.

The equations of interest in many simulations are driven by multiple physical processes, e.g., associated 
with the simulation of fluid-structure interaction, sub-grid-scale physics, or
chemical reaction terms. These processes have different dynamical characteristics, with some being slow
(e.g., transport) and some fast (e.g., stiff chemistry).
In addition, multiple spatial and temporal scales are associated with non-uniform  meshes, with boundary layers, 
with fast waves, and with the structure of the system  (e.g., the existence of jet fans).
The existence of fast and slow dynamics poses considerable challenges to the solution of the semi-discrete PDEs 
\eqref{eqn:ode} by explicit time stepping methods.
Specifically, due to the Courant-Friedrichs-Lewy (CFL) stability condition, the largest allowable step sizes 
are bounded above by the shortest (fastest) time scale in the system.

To avoid stability restrictions on the step size, implicit time integration methods are becoming widely used in 
the simulation of large-scale evolutionary PDEs  \eqref{eqn:ode}.   
Implicit time stepping requires the solution of large nonlinear systems of equations, coupling all variables in the model, 
at each time step. The associated computation and communication costs
constitute a major scalability bottleneck at best, and can quickly become infeasible for large systems. 

In this paper we examine, and extend, the Rosenbrock class of integration methods.  These are linearly-implicit methods which enjoy the benefit of requiring a fixed number of linear solves, as opposed to non-linear solves in the case of implicit Runge-Kutta methods.  A general $s$-stage Rosenbrock method reads \cite[Section IV.7]{Hairer_book_II}

%%%%%%%%%%%%%%%%%%%%%%
\begin{subequations}
\label{eqn:ros}
\begin{eqnarray}
    		y_{n+1} &=& y_n + \displaystyle\sum_{i=1}^s b_i\, k_i \,,
                \label{eqn:ros-solution}  \\
    		\qquad\left(\mathbf{I}_{N \times N} - h\, \gamma_{i,i}\, \mathbf{A}_n \right) \cdot k_i &=& h\,f\left(t_n + \alpha_i h, y_n+\displaystyle\sum_{j=1}^{i-1}\alpha_{i,j} \,k_j \right) 
		\label{eqn:ros-stage} \\
\nonumber && + h\,\mathbf{A}_n\, \displaystyle\sum_{j=1}^{i-1} \gamma_{i,j}\, k_j  + h^2\, \gamma_i\, f_t(t_n,y_n)	
		\,, ~~  i=1,\dots,s\,.
\end{eqnarray}
\end{subequations}
%%%%%%%%%%%%%%%%%%%%%%

Here $y_n \approx y(t_n)$ and $y_{n+1} \approx y(t_{n+1})$ are the numerical solutions computed by the method, and $h=t_{n+1}-t_n$ is the current time step.
The term $f_t = \partial f / \partial t$ is the partial time derivative of $f$ evaluated at the beginning of the current time step.
The method coefficients $\alpha_{i,j}$, $\gamma_{i,j}$, and $b_i$ are chosen such as to ensure the desired accuracy and stability properties. For convenience of notation one also defines
\[
\alpha_i = \sum_{j=1}^{i-1} \alpha_{i,j}\,, \quad \gamma_i = \sum_{j=1}^{i} \gamma_{i,j}\,, \quad \beta_{i,j} = \alpha_{i,j} + \gamma_{i,j}\,.
\]
In classical Rosenbrock methods the matrix $\mathbf{A}_n = \mathbf{J}_n \in \R^{N\times N}$, where  
\[
\mathbf{J}_n = f_y(t_n,y_n)
\]
is the Jacobian of $f$ evaluated at the beginning of the current time step. The vectors $k_i$ are the intermediate stage values; each of them is computed by solving an $N \times N$ linear system.
Typically all coefficients are chosen equal to each other, $\gamma_{i,i} = \gamma$ for $ i=1,\dots,s$, such that all stages share the same LU decomposition.
The order conditions of classical Rosenbrock methods rely on the assumptions that $\mathbf{A}_n$ is the exact ODE Jacobian, and that each of the linear systems \eqref{eqn:ros-stage} is solved exactly.

For many systems the exact Jacobian $\mathbf{J}_n$ can be both costly and difficult to obtain, e.g., due to the size of the application and the use of complex spatial discretization schemes.  
The class of Rosenbrock-W methods \cite[Section IV.7]{Hairer_book_II} has been derived for such situations. They have the same form \eqref{eqn:ros},
but the coefficients are selected such that the overall discretization order is preserved for any matrix $\mathbf{A}_n$, i.e., for arbitrary approximations of the Jacobian \cite{Rang_2005_ROW3}. 
The role of the matrix is to ensure numerical stability, and its choice dictates the type and amount of implicitness used in \eqref{eqn:ros}.  

%%%
\iffalse
{\mt Similar to the last comment, the discussion in the middle of page two correctly identifies a
difficulty in methods which rely on the exact formulation of the Jacobian matrix, in
particular for PDEs. This leads the reader to believe that the Rosenbrock-Krylov methods
avoid this problem, although the discussion of the methods in section 2 appears to demand
that one can build the explicit Krylov subspace using the Jacobian. It would be helpful if
the discussion early in the paper differentiated
between the need to compose the exact Jacobian as opposed to
only evaluating the Jacobian times a vector, as opposed to approximating the product of the
Jacobian times a vector (as in section 3.3). This would help make clear the overhead in the
current methods for PDE type applications.}
\fi
%%% 

Solving for the stage values $k_i$ \eqref{eqn:ros-stage} is the most expensive part of the integrator.  For large-scale systems direct methods such as $LU$ decomposition are not feasible, and iterative Krylov based methods, such as GMRES \cite{Saad,vanderVorst}, have been considered in the literature \cite{Weiner_1998_order,Wensch_2005_DAE,Weiner_1995_krylovW}.
Krylov based methods use only Jacobian-vector products and do not require storage, or even knowledge of, the full Jacobian.

The use of Krylov(-Newton) methods is a natural approach to speed up 
the solutions of the  linear/nonlinear systems of equations in implicit time integration \cite{Kelley_2004_Krylov,Keyes_2004_JFNK}.
Such integration methods have been successfully employed in real-life applications \cite{Grotowskya_1996,Schulze_2008_exp,Evans_2007_JFNK-SIMPLE}.
Krylov space solvers have been used in the implementation of implicit Runge-Kutta \cite{Jay_1999_SPARK,Dekker_2009_Krylov,Bouhamidi_2011_Krylov},
implicit linear multistep \cite{Hosseini_1999_MEBDF,Hindmarsh_2002_VODPK,Evans_2007_JFNK-SIMPLE},
and deferred correction \cite{Minion_2005_SD-Krylov} methods. Software for solving stiff ODEs and DAEs with this approach 
include  LSODKR \cite{Hindmarsh_1991_LSODKR}, LSODPK  \cite{Hindmarsh_1991_LSODPK},
VODPK \cite{Hindmarsh_2002_VODPK}, and DASPK \cite{Brown_1994_DASPK}.
In addition, Krylov space methods have been used for the exact integration of linear ODEs with source terms \cite{Botchev_2011_Krylov}, 
and to improve stability \cite{Botchev_2001_Krylov}.
Krylov space techniques have been used to accelerate convergence of deferred correction methods \cite{Minion_2006_accelerated,Minion_2007_dae}.

Of particular interest in this work is the use of Krylov methods in the context of Rosenbrock time integrators \eqref{eqn:ros}.
Classical Rosenbrock integrators are poor matrix free methods due to the explicit presence of the exact Jacobian matrix, and the approximate nature of Krylov based methods.  Rosenbrock-W methods are better suited for coupling with Krylov based solvers.
A family of matrix-free Rosenbrock methods, named Krylov-ROW, has been proposed in the 1990's  \cite{Weiner_1995_krylovW,Schmitt_1995_krylovW,Podhaisky_1997_krylovW,Weiner_1997_rowmap,Strehmel_1991_ROW-stiff}.
The control of linear solution accuracy in each stage is discussed in
\cite{Schmitt_1995_krylovW}, and preconditioning in \cite{Schmitt_1998_precondition}.
Order results for Krylov-ROW methods are studied in \cite{Weiner_1995_krylovW,Weiner_1998_order}. A multiple Arnoldi process
is proposed, where the Krylov space is enriched at each stage, such that the information
from previous stages is reused, and all the right hand side vectors belong to the subspace.
The order of the underlying Rosenbrock method is preserved under modest requirements on the Krylov space size, and
independent of the dimension of the ODE system. The implementation of methods of order four is done in the code ROWMAP \cite{Weiner_1997_rowmap},
where the error estimation and step size selection strategies are
inherited from the underlying Rosenbrock method.
The application of Krylov-ROW methods to index-1 DAEs \cite{Wensch_2005_DAE} reveals
that the Krylov space dimension needs to exceed the number of algebraic variables.
Krylov-ROW methods are therefore attractive in the case where the number of algebraic constraints is small
compared to the number of differential equations; or, by extension, where the dimension of
the stiff subspace is small. 
Novati \cite{Novati_2008_secantROW} presents a class of W-methods where the Jacobian is approximated using quasi-Newton-like
rank one updates, based on solution and function values in previous time steps. 
Periodic restarts are needed for stability, as the Jacobian approximation deteriorates with time.
A related family of methods are exponential integrators  \cite{Hochbruck_1997_exp}, which use matrix exponentials of the Jacobian as part
of the solution process, and evaluate matrix exponential times vector products via Krylov space methods \cite{Hochbruck_2009_exp,Caliaria_2009_implementation,Tokman_2006_EPI,Hochbruck_1998_exp,Hochbruck_2010_exp}.
%Order condition theories for exponential integrators are developed in \cite{Berland_2006_order,Butcher_2011_exp-order}.

In this paper we develop a new family of Rosenbrock-Krylov time stepping methods characterized by the lowest possible degree of implicitness that ensures stability. The new algorithms  are implicit in only the stiff subspace, which is captured by a Krylov space. Moreover, 
they perform only scalable operations such as Jacobian-vector products, and solve only small linear systems at each step.
A naive implementation of a matrix free Rosenbrock-W method requires construction of a Krylov subspace for each stage equation.  
The Rosenbrock-K methods proposed herein extend the framework of Rosenbrock-W methods, accounting for the Krylov subspace approximation of the linear system when constructing the order conditions of the integrator.  In this way we substantially reduce the number of required order conditions for a given order, thereby reducing the number of necessary stages of the method. 
The Rosenbrock-K methods require the construction of only a single Krylov subspace for the solution of all stage values. The dimension of this subspace need only be as large as the desired order of the method to ensure accuracy. 
%
%The class of Rosenbrock-K methods differs from Krylov-ROW family \cite{Weiner_1995_krylovW,Weiner_1998_order} in several important aspects. First, the Krylov space properties are an integral part of Rosenbrock-K order condition theory;
%this elegant approach ensures the desired orders of accuracy with much smaller subspaces than those required by the Krylov-ROW approach. Next, a single Krylov space is used across all stages; the implementation of Rosenbrock-K is much simpler, and considerably more scalable, than the implementation of Krylov-ROW, where the subspace needs to be extended at each subsequent stage. 

The new class of Rosenbrock-K methods differs from existing Rosenbrock schemes in several important aspects.  
Rosenbrock-K methods use approximate Jacobians similar to Rosenbrock-W family.
The Jacobian approximations have to be based on a Krylov subspace approach, similar to the Krylov-ROW family.
%Like Rosenbrock-W methods they allow for the use of an approximate Jacobian, however, similar to the Krylov-ROW family require that this approximation be based on a Krylov subspace approach.  
The primary benefit of Rosenbrock-K methods over Krylov-ROW stems from the integration of the Krylov space properties into the Rosenbrock-K order condition theory; this elegant approach ensures the desired orders of accuracy with much smaller subspaces than those required by the Krylov-ROW approach.  More importantly, a single Krylov space is used across all stages, and most operations are performed in this reduced space.  The implementation of Rosenbrock-K is thus much simpler, and considerably more scalable, than the implementation of Krylov-ROW, where the subspace needs to be extended at each subsequent stage.

The paper is laid out as follows: in Section \ref{sec:RK-methods} we present the framework of the proposed class of methods as well as the Krylov subspace approximation of the Jacobian used, in Section \ref{sec:orderconds} we extend the theory of order trees for Rosenbrock-w methods to our new Rosenbrock-Krylov methods as well as give details of how to construct these trees and the method order conditions from them, in Section \ref{sec:solvingorderconds} we construct two new Rosenbrock-Krylov integrators and outline a method for the solution of the order conditions to derive specific method coefficients, and in Section \ref{sec:numericalresults} we present some numerical results.

%%%%%%%%%%%%%%%%%%%%%%%%%%%%%%%%%%%%%%%%%%
\section{Formulation of Rosenbrock-Krylov methods}\label{sec:RK-methods}
%%%%%%%%%%%%%%%%%%%%%%%%%%%%%%%%%%%%%%%%%%

%
%{\mt At the beginning of Sect. 2, it would be useful to discuss the computational cost of the
%Arnoldi process, especially in the context of large matrices and parallelization.}
%

Rosenbrock-Krylov methods have the same form as Rosenbrock-W methods \eqref{eqn:ros}, but use a particular approximation $\mathbf{A}_n$ of $\mathbf{J}_n$.
We start the presentation with the case of autonomous systems.

Specifically, let $f_n=f(y_n)$ and consider the $M$-dimensional Krylov space
\begin{eqnarray}
\label{eqn:krylovspace}
	\K_M(\mathbf{J}_n\,,f_n) &=& \textnormal{span}\left\{\, f_n, \, \mathbf{J}_n\,f_n, \, \mathbf{J}_n^2\,f_n, \dots, \mathbf{J}_n^{M-1}\,f_n\, \right\} \\
\nonumber
	&=& \textnormal{span}\left\{v_1, v_2,  \dots, v_M\right\} \,.
\end{eqnarray}
An orthogonal basis $\{ v_i\}_{i=1,\dots,M}$ for $\K_M$ is constructed using a modified Arnoldi process \cite{vanderVorst}.
The Arnoldi iteration returns two valuable pieces of information: a matrix $\mathbf{V}_{n;M}$ whose columns are the orthonormal basis vectors of $\K_M$, and an upper Hessenberg matrix $\mathbf{H}_{n;M}$, such that 
\begin{equation}
\label{eqn:htoj}
\mathbf{V}_{n;M} = \left[ v_1,\dots v_M \right] \in \R^{N \times M}\,; \quad  
	\mathbf{H}_{n;M} = \mathbf{V}_{n;M}^T \, \mathbf{J}_n \, \mathbf{V}_{n;M} \in \R^{M \times M}\,.
\end{equation}

{\it The Rosenbrock-Krylov matrix $\mathbf{A}_n$ is the restriction of the full ODE Jacobian to the Krylov space:}
\begin{equation}
\label{eqn:ROK-approximation}
	\mathbf{A}_n = \mathbf{V}_{n;M}\, \mathbf{H}_{n;M}\,  \mathbf{V}_{n;M}^T =  \mathbf{V}_{n;M}\, \mathbf{V}_{n;M}^T\, \mathbf{J}_n\,  \mathbf{V}_{n;M}\,  \mathbf{V}_{n;M}^T\,.
\end{equation}
To obtain the stage vector $k_i$ we decompose it 
in the component residing in $\K_M$, and the component orthogonal to $\K_M$ 
\begin{equation}
\label{eqn:k-decomposition}
k_i = \underbrace{\mathbf{V}_{n;M}\, \lambda_i}_{\in \K_M} + \underbrace{\mu_i}_{\in \K_M^\perp} \,,  
\end{equation}
where the new vectors $\lambda_i$ and $\mu_i$ are defined by 
\[
\lambda_i = \mathbf{V}_{n;M}^T\, k_i \in \R^M\,, \quad  \mu_i = \left(\mathbf{I}_{N \times N}-\mathbf{V}_{n;M}\, \mathbf{V}_{n;M}^T\right)\, k_i \in \R^N\,.  
\]
We consider also the projections of the function values in \eqref{eqn:ros-stage} onto the Krylov space
\begin{equation*}
	F_i = f\left( y_n + \displaystyle\sum_{j=1}^{i-1}\alpha_{i,j}k_j\right)\,, \quad  \phi_i = \mathbf{V}_{n;M}^T\, F_i \in \R^M\,.
\end{equation*}

To construct a Rosenbrock-K method the Jacobian approximation \eqref{eqn:ROK-approximation}  and the decomposition \eqref{eqn:k-decomposition} 
are used in the stage formulation \eqref{eqn:ros-stage} to obtain
\[
\left(\mathbf{I}_{N \times N} - h\, \gamma\, \mathbf{V}_{n;M}\, \mathbf{H}_{n;M}\, \mathbf{V}_{n;M}^T \right) \cdot \left(\mathbf{V}_{n;M}\lambda_i + \mu_i\right)  = h\,F_i + h\,\mathbf{V}_{n;M}\, \mathbf{H}_{n;M}\, \mathbf{V}_{n;M}^T\, \displaystyle\sum_{j=1}^{i-1} \gamma_{i,j}\,\left(\mathbf{V}_{n;M}\lambda_j + \mu_j\right) \,.
\]
Using the facts that $\mathbf{V}_{n;M}^T\mathbf{V}_{n;M} = \mathbf{I}_{M \times M}$ and $\mathbf{V}_{n;M}^T\mu_i = 0$ the equation can be written as
\begin{eqnarray}
\label{eqn:ROK-stage-decomposed}
\mu_i + \mathbf{V}_{n;M}\left( \mathbf{I}_{M \times M} - h\, \gamma\, \mathbf{H}_{n;M} \right)\, \lambda_i 
&=& h\,F_i + h\,\mathbf{V}_{n;M}\, \mathbf{H}_{n;M}\, \displaystyle\sum_{j=1}^{i-1} \gamma_{i,j}\, \lambda_j \,.
\end{eqnarray}
Multiplying both sides of \eqref{eqn:ROK-stage-decomposed} by $\mathbf{V}_{n;M}^T$ leads to
the following reduced stage equation 
\begin{equation}
\label{eqn:reduced-stage-lambda}
	\left( \mathbf{I}_{M \times M} - h\gamma \mathbf{H}_{n;M} \right) \lambda_i = h\phi_i + h \mathbf{H}_{n;M}\displaystyle\sum_{j=1}^{i-1} \gamma_{i,j}\lambda_j\,.
\end{equation}
Multiplying both sides of \eqref{eqn:ROK-stage-decomposed} by $\mathbf{I}_{N \times N}- \mathbf{V}_{n;M}\mathbf{V}_{n;M}^T$ gives
\begin{equation}
\label{eqn:reduced-stage-mu}
\mu_i = h\,\left( F_i - \mathbf{V}_{n;M}\, \phi_i\right)\,.
\end{equation}
The system \eqref{eqn:reduced-stage-lambda} is of size $M \times M$, where $M \ll N$, and can be inverted through the use of a direct method.  The full stage values can now be recovered from \eqref{eqn:k-decomposition}, \eqref{eqn:reduced-stage-lambda}, and \eqref{eqn:reduced-stage-mu} as
\begin{equation}
\label{eqn:ROK-full-stage}
	k_i = \mathbf{V}_{n;M}\, \lambda_i + h\, \left(F_i - \mathbf{V}_{n;M}\phi_i\right) \,.
\end{equation}
%
\iffalse
An autonomous Rosenbrock-K step from $t_n$ to $t_{n+1}$ is summarized in Algorithm \ref{ROK-autonomous-step}.

\begin{algorithm}[ht]
\caption{One step of an autonomous Rosenbrock-K integrator}\label{ROK-autonomous-step}
\begin{algorithmic}[1]
   \State Compute $\mathbf{H}_{n;M}$ and $\mathbf{V}_{n;M}$ \eqref{eqn:htoj} using the $N$-dimensional Arnoldi process \cite{vanderVorst}
   \For{$i = 1,\dots,s$}\Comment{For each stage, in succession}
      %\State 
      \begin{eqnarray*}
       F_i &=& f\left( y_n + \displaystyle\sum_{j=1}^{i-1}\alpha_{i,j}k_j\right) \\
       \phi_i &=& \mathbf{V}_{n;M}^T\, F_i \\
      \lambda_i &=& \left( \mathbf{I}_{M \times M} - h\gamma \mathbf{H}_{n;M}\right)^{-1}\, \left( h\phi_i + h \mathbf{H}_{n;M} \displaystyle\sum_{j=1}^{i-1} \gamma_{i,j} \lambda_j\right) \\
       k_i &=& \mathbf{V}_{n;M}\, \lambda_i + h\, (F_i - \mathbf{V}_{n;M}\, \phi_i)
     \end{eqnarray*}
   \EndFor
   \State $y_{n+1} = y_n + \displaystyle\sum_{i=1}^s b_i k_i$
\end{algorithmic}
\end{algorithm}
\fi

We next consider the formulation of Rosenbrock-K methods for non-autonomous problems \eqref{eqn:ode}.  
With the extended state and function
\begin{equation}
\label{eqn:naextend}
	\widehat{y}(t) = \left[\begin{array}{c} y(t) \\ t \end{array}\right] \in \R^{N+1}\,, \quad  \widehat{f}(\widehat{y}) = \left[ \begin{array}{c} f(t,y) \\ 1 \end{array}\right]  \in \R^{N+1}\,,
\end{equation}
the general ODE \eqref{eqn:ode} can be formulated  as autonomous system
\begin{equation}
\label{eqn:autode}
	\frac{d\widehat{y}}{dt} = \widehat{f}(\widehat{y})\,, \quad  \widehat{y}(t_0) = \left[ \begin{array}{c} y \\ t_0 \end{array} \right], \quad  t_0 \leq t \leq t_F, \quad  \widehat{y}(t) \, , \widehat{f}(\widehat{y})\,.
\end{equation} 
Non-autonomous Rosenbrock-K integrators are constructed using this technique.  The Jacobian of the extended right hand side function is
\begin{equation}
\label{eqn:jacextend}
	\widehat{\mathbf{J}} = \left[ \begin{array}{cc} f_y & f_t \\ 0 & 0 \end{array} \right].
\end{equation}
An extended Krylov space $\K_M(\widehat{\mathbf{J}}_n\,,\widehat{f}_n)$ is constructed using matrix-vector products of the form
\begin{equation}
\label{eqn:mfjvp}
\widehat{\mathbf{J}}_n \begin{bmatrix} \zeta^T ~ \xi \end{bmatrix}^T = \begin{bmatrix} \mathbf{J}_n \, \zeta + f_t(t_n,y_n)\, \xi \\ 0 \end{bmatrix}\,,
\quad \zeta \in \R^N\,,\quad \xi \in \R\,.
\end{equation}
The extended $(N+1)$-dimensional Arnoldi iterations produce the matrices $\widehat{\mathbf{V}}_{n;M}$ and $\mathbf{H}_{n;M}$ such that
\begin{equation}
\label{eqn:HVw}
\widehat{\mathbf{V}}_{n;M} = \begin{bmatrix} \mathbf{V}_{n;M} \\ w^T  \end{bmatrix} \in \R^{(N+1) \times M}\,, \quad 
\mathbf{V}_{n;M} \in \R^{N \times M}\,, \quad w \in \R^M\,,
\end{equation}
and
\[
	\mathbf{H}_{n;M} = \widehat{\mathbf{V}}_{n;M}^T \, \widehat{\mathbf{J}}_n \, \widehat{\mathbf{V}}_{n;M}
	= \mathbf{V}_{n;M}^T \, \mathbf{J}_n \, \mathbf{V}_{n;M} + \mathbf{V}_{n;M}^T \, f_t(t_n,y_n) \, w^T \in \R^{M \times M}\,.
\]
This modified Arnoldi iteration proceeds as follows. 
%%%%%%%%%%%%%%%%%%%%%%
\begin{algorithm}[ht]
\caption{Modified Arnoldi iteration}
\begin{algorithmic}[l]
   \State $\beta = \left\Vert \begin{bmatrix} f_n^T ~ 1 \end{bmatrix}^T \right\Vert~$; $\quad w_1 = 1/\beta$; $\quad  v_1 = f_n/\beta~$.
   \For{$i=1,\dots,M$}
       
       \State $\displaystyle \zeta =  \mathbf{J}_n \, v_i + f_t(t_n,y_n)\, w_i$  ; $\quad$
        $\displaystyle \xi    = 0$ ; $\quad$
        $\displaystyle \tau = \Vert \zeta \Vert$ 
      \For{$j = 1,\dots,i$} 
         \State $H_{j,i} = \left\langle \zeta , v_j \right\rangle + \xi\, w_j~ $; $\quad$
                   $\displaystyle \zeta = \zeta  - H_{j,i}\, v_j$; $\quad$
                   $\displaystyle \xi    = \xi -  H_{j,i}\, w_j $
      \EndFor
      \If{ $\left\Vert \begin{bmatrix} \zeta^T ~ \xi \end{bmatrix}^T \right\Vert / \tau \leq \kappa$}
         \For{$j = 1,\dots,i$}
            \State $\displaystyle \rho = \left\langle \zeta, v_j \right\rangle + \xi\, w_j $; $\quad$
                      $\displaystyle \zeta = \zeta - \rho\, v_j$; $\quad$
                      $\displaystyle \xi = \xi - \rho\, w_j$; $\quad$
                      $\displaystyle H_{j,i} = H_{j,i} + \rho$
         \EndFor
      \EndIf
      \State $\displaystyle H_{i+1,i} = \left\Vert \begin{bmatrix} \zeta^T ~ \xi \end{bmatrix}^T \right\Vert$; $\quad$
                $\displaystyle v_{i+1} = \zeta/H_{i+1,i}$; $\quad$
                $\displaystyle w_{i+1} = \xi/H_{i+1,i}$; $\quad$
   \EndFor
\end{algorithmic}
\end{algorithm}

The non-autonomous Rosenbrock-K integrator is obtained by applying the autonomous step (described above) to the extended system \eqref{eqn:autode}, and decoupling the state and time variables. The procedure is summarized in Algorithm \ref{ROK-nonautonomous-step}. The autonomous version is obtained by letting $w=0$.

\begin{algorithm}[ht]
\caption{One step of a non-autonomous Rosenbrock-K integrator}\label{ROK-nonautonomous-step}
\begin{algorithmic}[1]
   \State Compute $\mathbf{H}_{n;M}$, $\mathbf{V}_{n;M}$, and $w$ \eqref{eqn:HVw} using the $(N+1)$-dimensional Arnoldi process.
   \For{$i = 1,\dots,s$}\Comment{For each stage, in succession}
      %\State 
      \begin{eqnarray*}
       F_i &=& f\left( t_n + \alpha_i h\,,\, y_n + \displaystyle\sum_{j=1}^{i-1}\alpha_{i,j}k_j\right) \\
       \phi_i &=& \mathbf{V}_{n;M}^T\, F_i + w \\
      \lambda_i &=& \left( \mathbf{I}_{M \times M} - h\gamma \mathbf{H}_{n;M}\right)^{-1}\, \left( h\phi_i + h \mathbf{H}_{n;M} \displaystyle\sum_{j=1}^{i-1} \gamma_{i,j} \lambda_j\right) \\
       k_i &=& \mathbf{V}_{n;M}\, \lambda_i + h\, (F_i - \mathbf{V}_{n;M}\, \phi_i)
     \end{eqnarray*}
   \EndFor
   \State $y_{n+1} = y_n + \displaystyle\sum_{i=1}^s b_i k_i$
\end{algorithmic}
\end{algorithm}

%\begin{equation*} 
%\begin{array}{l}
%\mbox{\textbf{One step of the non-autonomous Rosenbrock-Krylov integrator}} \\ 
%\mbox{Compute $H \in \R^{M \times M}$ and $V \in \R^{N+1 \times M}$ using the modified Arnoldi process\cite{vanderVorst}} \\
%\widehat{V} = V(1:N,1:M) \\
%\widehat{w} = V(N+1,1:M) \\
%\mbox{\textbf{for} i = 1:s}\\
%\quad  \begin{array}{l}
%F_i = f\left(t_n + \displaystyle\sum_{j=1}^{i-1}\alpha_{i,j}h, \,  y_n + \displaystyle\sum_{j=1}^{i-1}\alpha_{i,j}k_j\right) \\
%\phi_i = \widehat{V}^T F_i + \widehat{w}^T \\
%\lambda_i = \left( I_{M \times M} - h\gamma H\right)^{-1} \left( h\phi_i + hH \displaystyle\sum_{j=1}^{i-1} \gamma_{i,j} \lambda_j\right) \\
%k_i = \widehat{V}\lambda_i + h(F_i - \widehat{V}\phi_i)
%\end{array} \\
%\mbox{\textbf{endfor} i}\\
%y_{n+1} = y_n + \displaystyle\sum_{i=1}^s b_i k_i \\
%t_{n+1} = t_n + h
%\end{array}
%\end{equation*}
%%%%%%%%%%%%%%%%%%%%%%%%

%%%%%%%%%%%%%%%%%%%%%%%%
\section{Order Conditions}\label{sec:orderconds}
%%%%%%%%%%%%%%%%%%%%%%%%

The accuracy theory is based on matching the Taylor series of the numerical solution and of the exact solution, up to some specified order.
Butcher-trees \cite[Section IV.7]{Hairer_book_I} are a well accepted method of representing individual terms in the Taylor series expansions.
The derivation of order conditions for Rosenbrock-K methods is an extension of the framework used to derive order conditions for Rosenbrock-W methods.  The existing theory for Rosenbrock-W methods is based on the use of $TW$-trees, a subclass of $P$-trees \cite[Section IV.7]{Hairer_book_II}.  $P$-trees themselves are an extension of the set $T$ of Butcher-trees that allow for two different colors of the nodes. We have the following definition  \cite[Section IV.7]{Hairer_book_II}:
\[ 
TW = \left\{ \begin{array}{cl}P\textrm{-trees:} & \textrm{end vertices are meagre, and} \\
				 & \textrm{fat vertices are singly branched} \end{array} \right\} 
\]
In the context of Rosenbrock-K and Rosenbrock-W methods, a meagre node represents an actual derivative of $f$ coming from the first term on the right of equation \eqref{eqn:ros-stage}, while a meagre node is an appearance of the approximate Jacobian $\mathbf{A}_n$ coming from the second term on the right of equation \eqref{eqn:ros-stage}.  Each tree represents a single elementary differential in the Taylor series of either the exact or numerical solutions of the ODE.  

Figure \ref{fig:ROWtrees} shows all TW-trees and Rosenbrock-W conditions for up to order three \cite[Section IV.7]{Hairer_book_II}. 
The correspondence between the $TW$-trees and elementary differentials, and Rosenbrock-W order conditions, is summarized next:
\begin{itemize}
\item For the elementary differentials in Figure \ref{fig:ROWtrees} superscripts represent component indices, and subscripts represent indices of variables with respect to which partial derivatives are taken. For example, $f^J$ is the $J$-th component of $f$, and $f^J_K f^K = \sum_K \partial f^J/\partial y^K \cdot f^K$.
 \item A meagre node represents a derivative of $f$.
 \item The order of the $f$ derivative equals the number of children the meagre node has.
 \item A fat node represents the appearance of the approximate Jacobian matrix, $\mathbf{A}_n$, in the elementary differential.
\end{itemize}
The correspondence between the $TW$-trees and Rosenbrock-W order conditions, is as follows:
\begin{itemize}
\item For the order conditions in Figure \ref{fig:ROWtrees} the summations apply to all repeated indices in the expression.
 \item For Rosenbrock-W methods:
 \quad \begin{itemize}
	\item an edge originating from meagre node $j$ to another node $k$ gives $\alpha_{j,k}$,
	\item an edge originating from fat node $j$ to another node $k$ gives $\gamma_{j,k}$.
 \end{itemize}
 \item For classical Rosenbrock methods
 \quad \begin{itemize}
	\item an edge connecting meagre node $j$, having multiple children, to a meagre node $k$ gives $\alpha_{j,k}$,
	\item an edge connecting meagre node $j$, having a single child, to a meagre node $k$ gives $\beta_{j,k}$. 
 \end{itemize}
\end{itemize}

The exact solution is represented by trees containing only meagre nodes, since the approximate Jacobian matrix never appears in its series expansion.  For this reason Rosenbrock-W methods have two sets of order conditions: those arising from trees containing only meagre nodes, and those arising from trees containing at least one fat node.  Trees containing fat nodes do not correspond to any trees in the exact solution and, as seen in Figure \ref{fig:ROWtrees}, the corresponding coefficients are set to zero  \cite{Hairer_book_II}.

%%%%%%%%%%%%%%%%%%%%%%%%
\begin{figure}[htp]
  \caption{TW-trees and Rosenbrock-W conditions up to order three \cite[Section IV.7]{Hairer_book_II}. }
  \label{fig:ROWtrees}
  \begin{center}
  \begin{tabular}{|c|c|c|lcr|}
  \hline
  $a$ & \begin{tikzpicture}[scale=.5]
      [meagre/.style={circle,draw, fill=black!100,thick},
      fat/.style={circle,draw,thick}]
      \node[circle,draw, fill=black!100,thick] (j) at (0,0) [label=right:$j$] {};
  \end{tikzpicture} & $ f^J $ & $\sum b_j$ & $=$ & $1$ \\
\hline
 $b_1$ & \begin{tikzpicture}[scale=.5]
      [meagre/.style={circle,draw, fill=black!100,thick},
      fat/.style={circle,draw,thick}]
      \node[circle,draw, fill=black!100,thick] (j) at (0,0) [label=right:$j$] {};
      \node[circle,draw, fill=black!100,thick] (k) at (1,1) [label=right:$k$] {};
      \draw[-] (j) -- (k);
  \end{tikzpicture} & $ f^J_Kf^K $ & $\sum b_j \alpha_{j,k}$ & $=$ & $1/2$ \\
$b_2$  & \begin{tikzpicture}[scale=.5]
      [meagre/.style={circle,draw, fill=black!100,thick},
      fat/.style={circle,draw,thick}]
      \node[circle,draw, thick] (j) at (0,0) [label=right:$j$] {};
      \node[circle,draw, fill=black!100,thick] (k) at (1,1) [label=right:$k$] {};
      \draw[-] (j) -- (k);
  \end{tikzpicture} & $ \mathbf{A}_{JK}f^K $ & $\sum b_j \gamma_{j,k}$ & $=$ & $0$ \\
\hline
$c$  & \begin{tikzpicture}[scale=.5]
      [meagre/.style={circle,draw, fill=black!100,thick},
      fat/.style={circle,draw,thick}]
      \node[circle,draw, fill=black!100,thick] (j) at (1,0) [label=right:$j$] {};
      \node[circle,draw, fill=black!100,thick] (k) at (2,1) [label=right:$k$] {};
      \node[circle,draw, fill=black!100,thick] (l) at (0,1) [label=right:$l$] {};
      \draw[-] (j) -- (k);
      \draw[-] (j) -- (l);
  \end{tikzpicture} & $ f^J_{KL}f^Kf^L $ & $\sum b_j \alpha_{j,k} \alpha_{j,l}$ & $=$ & $1/3$ \\
\hline
$d_1$  & \begin{tikzpicture}[scale=.5]
      [meagre/.style={circle,draw, fill=black!100,thick},
      fat/.style={circle,draw,thick}]
      \node[circle,draw, fill=black!100,thick] (j) at (0,0) [label=right:$j$] {};
      \node[circle,draw, fill=black!100,thick] (k) at (1,1) [label=right:$k$] {};
      \node[circle,draw, fill=black!100,thick] (l) at (0,2) [label=right:$l$] {};
      \draw[-] (j) -- (k);
      \draw[-] (k) -- (l);
  \end{tikzpicture} & $ f^J_Kf^K_Lf^L $ & $\sum b_j \alpha_{j,k}\alpha_{k,l}$ & $=$ & $1/6$ \\
$d_2$ & \begin{tikzpicture}[scale=.5]
      [meagre/.style={circle,draw, fill=black!100,thick},
      fat/.style={circle,draw,thick}]
      \node[circle,draw, thick] (j) at (0,0) [label=right:$j$] {};
      \node[circle,draw, fill=black!100,thick] (k) at (1,1) [label=right:$k$] {};
      \node[circle,draw, fill=black!100,thick] (l) at (0,2) [label=right:$l$] {};
      \draw[-] (j) -- (k);
      \draw[-] (k) -- (l);
  \end{tikzpicture} & $ \mathbf{A}_{JK}f^K_Lf^L $ & $\sum b_j \gamma_{j,k}\alpha_{k,l}$ & $=$ & $0$ \\
$d_3$  & \begin{tikzpicture}[scale=.5]
      [meagre/.style={circle,draw, fill=black!100,thick},
      fat/.style={circle,draw,thick}]
      \node[circle,draw, fill=black!100,thick] (j) at (0,0) [label=right:$j$] {};
      \node[circle,draw, thick] (k) at (1,1) [label=right:$k$] {};
      \node[circle,draw, fill=black!100,thick] (l) at (0,2) [label=right:$l$] {};
      \draw[-] (j) -- (k);
      \draw[-] (k) -- (l);
  \end{tikzpicture} & $ f^J_K\mathbf{A}_{KL}f^L $ & $\sum b_j \alpha_{j,k}\gamma_{k,l}$ & $=$ & $0$ \\
$d_4$  & \begin{tikzpicture}[scale=.5]
      [meagre/.style={circle,draw, fill=black!100,thick},
      fat/.style={circle,draw,thick}]
      \node[circle,draw, thick] (j) at (0,0) [label=right:$j$] {};
      \node[circle,draw, thick] (k) at (1,1) [label=right:$k$] {};
      \node[circle,draw, fill=black!100,thick] (l) at (0,2) [label=right:$l$] {};
      \draw[-] (j) -- (k);
      \draw[-] (k) -- (l);
  \end{tikzpicture} & $ A_{JK}\mathbf{A}_{KL}f^L $ & $\sum b_j \gamma_{j,k}\gamma_{k,l}$ & $=$ & $0$ \\
\hline
  \end{tabular}
  \end{center}
\end{figure}
%%%%%%%%%%%%%%%%%%%%%%%%

In order to build the relevant trees for Rosenbrock-K methods we need to have a closer look at the properties of the Jacobian approximation \eqref{eqn:ROK-approximation}.

%%%%%%%%%%%%%%%%%%%%%%%%
\begin{lemma}[Property of the Rosenbrock-Krylov approximate Jacobian \eqref{eqn:ROK-approximation}]\label{lemma:ROK-matrix} 
For any $0 \le k \le M-1$ it holds that
\[
\mathbf{A}_n^k\, f_n = \mathbf{J}_n^k\, f_n\,,
\]
where $M=dim(\K_M)$.
\end{lemma}
%%%%%%%%%%%%%%%%%%%%%%%%
%%%%%%%%%%%%%%%%%%%%%%%%
\begin{proof}
Recall that $\mathbf{V}_{n;M} \, \mathbf{V}_{n;M}^T$ is the orthogonal projector onto $\mathcal{K}_M$.
If a vector $u$ is in the Krylov space $\mathcal{K}_M$, its orthogonal projection onto $\mathcal{K}_M$ is the vector itself:
\[
u \in \mathcal{K}_M \quad \Rightarrow \quad \mathbf{V}_{n;M} \, \mathbf{V}_{n;M}^T \, u = u\,.
\]

The proof of the Lemma is by finite induction. As the base case we have that 
\[
\mathbf{A}_n^0\, f_n = \mathbf{J}_n^0\, f_n = f_n\,.
\]
%
%For the first power
%\[ 
%\mathbf{A}_n\, f_n = \mathbf{V}_{n;M} \, \mathbf{V}_{n;M}^T \, \mathbf{J}_n \, \mathbf{V}_{n;M} \, \mathbf{V}_{n;M}^T\, f_n\,.  
%\]
%If $M \geq 1$ then $f_n \in \K_M$ and
%\[ 
% \mathbf{V}_{n;M} \, \mathbf{V}_{n;M}^T\, f_n = f_n \quad \Rightarrow \quad
%\mathbf{A}_n\, f_n = \mathbf{V}_{n;M} \, \mathbf{V}_{n;M}^T \, \mathbf{J}_n \, f_n\,.
%\]
%%
%Similarly, if $M \geq 2$ then $\mathbf{J}_n \, f_n \in \K_M$ and
%\[ 
%\mathbf{V}_{n;M} \, \mathbf{V}_{n;M}^T \, \mathbf{J}_n \, f_n  \quad \Rightarrow \quad
%\mathbf{A}_n\, f_n =  \mathbf{J}_n \, f_n\,.
% \]
 %
Next we assume that $\mathbf{A}_n^{i-1}\, f_n = \mathbf{J}_n^{i-1}\, f_n$ for some $i \leq M-1$ and will show that $\mathbf{A}_n^i\, f_n = \mathbf{J}_n^i\, f_n$.
By the definition of the approximate Jacobian \eqref{eqn:ROK-approximation} and our assumption it holds that
\[ 
\mathbf{A}_n^i\, f_n = \mathbf{A}_n\cdot \mathbf{A}_n^{i-1}\, f_n  = \mathbf{A}_n\cdot \mathbf{J}_n^{i-1}\, f_n= \mathbf{V}_{n;M} \, \mathbf{V}_{n;M}^T \, \mathbf{J}_n \, \mathbf{V}_{n;M} \, \mathbf{V}_{n;M}^T\cdot  \mathbf{J}_n^{i-1}\, f_n\,.  
 \]
Since $M \geq i$ we have that $\mathbf{J}_n^{i-1} f_n \in \K_M$, and 
\[ 
 \mathbf{V}_{n;M} \, \mathbf{V}_{n;M}^T\, \mathbf{J}_n^{i-1}f_n = \mathbf{J}_n^{i-1} f_n \quad \Rightarrow \quad
\mathbf{A}_n^i\, f_n = \mathbf{V}_{n;M} \, \mathbf{V}_{n;M}^T \, \mathbf{J}_n \cdot  \mathbf{J}_n^{i-1} f_n  = \mathbf{V}_{n;M} \, \mathbf{V}_{n;M}^T \,  \mathbf{J}_n^{i} f_n\,.
\]
Since $M \geq i+1$ we have that $\mathbf{J}_n^{i} f_n \in \K_M$ and
\[ 
 \mathbf{V}_{n;M} \, \mathbf{V}_{n;M}^T\, \mathbf{J}_n^{i}f_n = \mathbf{J}_n^{i} f_n \quad \Rightarrow \quad
\mathbf{A}_n^i\, f_n = \mathbf{J}_n^{i} f_n\,.
\]
\qquad
\end{proof}
%%%%%%%%%%%%%%%%%%%%%%%%

%%%%%%%%%%%%%%%%%%%%%%%%
\begin{lemma}[Property of elementary differentials using the approximation \eqref{eqn:ROK-approximation}]\label{lemma:ROK-differentials} 
When the Rosenbrock-Krylov matrix approximation  \eqref{eqn:ROK-approximation} is used, all {\it linear} TW-trees of order $k \le M$ correspond to a single elementary differential, regardless of the color of their nodes.  
\end{lemma}
%%%%%%%%%%%%%%%%%%%%%%%%
%%%%%%%%%%%%%%%%%%%%%%%%
\begin{proof}
In a linear tree each node has only one child.
A linear TW-tree with a fat root can be described by the sequence of its nodes, starting from the root. For example, the structure of a linear tree where the first $\nu_1$ nodes from the root are fat, followed by $\mu_1$ meagre nodes, etc. is described by the sequence $(\circ^{\nu_1} \bullet^{\mu_1} \dots \circ^{\nu_p} \bullet^{\mu_p})$ with $\mu_p>1$ (since the leaf is a meagre node). 

Consider now a tree of order $k = \nu_1 + \mu_1 + \dots + \nu_p + \mu_p \le M$. The corresponding elementary differential has the form $\mathbf{A}_n^{\nu_1} \mathbf{J}_n^{\mu_1} \dots \mathbf{A}_n^{\nu_p} \mathbf{J_n}^{\mu_p-1} f_n$. Repeated applications of Lemma \ref{lemma:ROK-matrix} reveal that
\begin{eqnarray*}
\mathbf{A}_n^{\nu_1} \mathbf{J}_n^{\mu_1} \dots \mathbf{A}_n^{\nu_p} \underbrace{\mathbf{J_n}^{\mu_p-1} f_n}_{=\mathbf{A}_n^{\mu_p-1} f_n} &=& \mathbf{A}_n^{\nu_1} \mathbf{J}_n^{\mu_1} \dots \underbrace{\mathbf{A}_n^{\nu_p+\mu_p-1} f_n}_{=\mathbf{J}_n^{\nu_p+\mu_p-1} f_n} = \dots
= \mathbf{J_n}^{k-1} f_n\,.
\end{eqnarray*}
Consequently, any linear TW-tree of order $k \le M$ has the same differential as the linear tree with only meagre nodes $(\bullet^{k})$.
\qquad
\end{proof}

An important consequence of Lemma \ref{lemma:ROK-differentials} is that if a linear TW sub-tree with $k \le M$ nodes has a fat root, 
the corresponding differential is the same as for the linear tree with only meagre nodes $(\bullet^{k})$. This observation allows us to essentially ``recolor'' linear TW sub-trees with a fat root (i.e., group them in classes of equivalence). This leads to the following.

\begin{definition}[TK-trees]
\begin{eqnarray*}
TK &=& \left\{ TW\textrm{-trees:}  \textrm{ no linear sub-tree has a fat root}  \right\} \,; \\
TK(k) &=& \left\{ TW\textrm{-trees:}  \textrm{ no linear sub-tree of order} \right.\\
                                 && \left. \qquad \qquad \textrm{ smaller than or equal to }   k  \textrm{ has a fat root}  \right\} \,.
\end{eqnarray*}
\end{definition}
For $t \in TK$ let $\rho(t)$ define the number of vertices.
We denote by $t={_\bullet}[t_1,\dots,t_m]$ a TK-tree tree with a meagre root linking the subtrees $t_1,\dots,t_m$, and by 
$t={_\circ}[t_1]$ a TK-tree with a fat root to which the subtree $t_1$ is connected. A special case is the single node tree ${_\bullet}[\,]$.
The elementary differentials associated with TK-trees are the same as those of TW-trees, \cite[Definition 7.5]{Hairer_book_II}.
The TK-tree coefficients are constructed recursively as follows.
%
%\begin{itemize}
%       \item $\phi_j(\tau) = 1$ for $\tau=\bullet$
%	\item an edge connecting a meagre node $j$, having multiple children, to a child $k$ gives $\alpha_{j,k}$,
%	$\sum_{k_1,\dots,k_m} \alpha_{j,k_m}\, \phi_{k_1}(t_1) \dots \phi_{k_m}(t_m)$ if $t={_\bullet}[t_1,\dots,t_m]$,  $m \ge 2$
%	\item an edge connecting a meagre node $j$, having a single child, to its child $k$ gives $\beta_{j,k}$, 
%	$\sum_k \beta_{j,k}\, \phi_k(t_1)$ if $t={_\bullet}[t_1]$
% 	\item an edge connecting a fat node $j$ to its child $k$ gives $\gamma_{j,k}$.
%	$\sum_k \gamma_{j,k}\, \phi_k(t_1)$ if $t={_\circ}[t_1]$
%\end{itemize}
%
\begin{definition}[Coefficients of TK-trees] For $t \in TK$
\[
\phi_j(t) = \left\{
\begin{array}{lll}
\displaystyle 1 & \mbox{if} & t={_\bullet}[\,] \\
\displaystyle \sum_{k_1,\dots,k_m} \alpha_{j,k_m}\, \phi_{k_1}(t_1) \dots \phi_{k_m}(t_m) & \mbox{if} & t={_\bullet}[t_1,\dots,t_m],~ m \ge 2 \\
\displaystyle \sum_k \beta_{j,k}\, \phi_k(t_1) & \mbox{if} & t={_\bullet}[t_1] \\
\displaystyle \sum_k \gamma_{j,k}\, \phi_k(t_1) & \mbox{if} & t={_\circ}[t_1] 
\end{array}
\right.
\]
\end{definition}

%%%%%%%%%%%%%%%%%%%%%%%%%%%%%%%%%%%%%
\subsection{Rosenbrock-K methods of type 1}
%%%%%%%%%%%%%%%%%%%%%%%%%%%%%%%%%%%%%

\begin{definition}
A Rosenbrock-K method of type 1 is given by Algorithm \ref{ROK-nonautonomous-step} and uses an underlying Krylov subspace given by \eqref{eqn:krylovspace}.
\end{definition}

\begin{subequations}
\begin{theorem}[Order conditions for Rosenbrock-K methods]\label{thm:ROK1-conditions}
A Rosenbrock-K method of type 1 has order $p$ iff the underlying Krylov space \eqref{eqn:krylovspace} has dimension $M \ge p$, and the following order conditions hold:
\label{eqn:ROK-conditions}
\begin{eqnarray}
\label{eqn:ROK-condition-T}
\sum_j b_j\, \phi_j(t) = \frac{1}{\gamma(t)} \quad \forall\, t \in T ~~ \mbox{with } \rho(t) \le p\,, \\
\label{eqn:ROK-condition-TW}
\sum_j b_j\, \phi_j(t) = 0 \quad \forall\, t \in TK\backslash T ~~ \mbox{with } \rho(t) \le p\,.
\end{eqnarray}
Here $\rho(t)$ is the number of vertices of the tree $t$, and $\gamma(t)$ is the ``product of $\rho(t)$ and all orders of the trees which appear, if the roots, one after another, are removed from $t$'' \cite[Section II.2]{Hairer_book_I}.
\end{theorem}
\begin{proof}
The proof follows from our discussion and from the order conditions of Rosenbrock-W methods \cite[Theorem 7.7]{Hairer_book_II}.
\qquad
\end{proof}

\begin{theorem}[Order conditions for Rosenbrock-K methods with smaller Krylov space]\label{thm:ROK1-conditions-small}
A Rosenbrock-K method of type 1 with an underlying Krylov space \eqref{eqn:krylovspace} of dimension $M < p$ has order $p$ iff condition
\eqref{eqn:ROK-condition-T} holds, and, in addition:
\begin{eqnarray}
%\sum_j b_j\, \phi_j(t) = \frac{1}{\gamma(t)} \quad \forall\, t \in T ~~ \mbox{with } \rho(t) \le p\,, \\
\label{eqn:ROK-condition-TWM}
\sum_j b_j\, \phi_j(t) = 0 \quad \forall\, t \in TK(M)\backslash T ~~ \mbox{with } \rho(t) \le p\,.
\end{eqnarray}
\end{theorem}
\end{subequations}
\begin{proof}
Follows immediately from Theorem \ref{thm:ROK1-conditions}.
\qquad
\end{proof}

\begin{remark}
The number of required order conditions for Rosenbrock-K methods is substantially smaller than the number of order conditions for Rosenbrock-W methods.
\end{remark}

Figure \ref{fig:ROWtrees} reveals that all TW-trees up to order three containing a fat root are linear, and so every tree containing a fat node can be recolored to contain only meagre nodes.  Thus the order conditions for Rosenbrock-K methods are the same as those for classical Rosenbrock methods for up to order three (while Rosenbrock-W methods need four additional conditions).  Figure \ref{fig:ROKtrees} shows the TK-trees and order conditions for up to order four; Rosenbrock-K methods  require only a single extra order condition for order four (while Rosenbrock-W methods require seventeen additional conditions). Finally, Figure \ref{fig:extraROK} shows
the four additional TK-trees and Rosenbrock-K conditions needed for order five.

%%%%%%%%%%%%%%%%%%%%%%%%
\begin{figure}[htp]
  \label{fig:ROKtrees}
  \caption{TK-trees and Rosenbrock-Krylov conditions up to order four.}
  \begin{center}
  \begin{tabular}{|c|c|c|lcr|}
\hline
  $a$ &
  \begin{tikzpicture}[scale=.5]
      [meagre/.style={circle,draw, fill=black!100,thick},
      fat/.style={circle,draw,thick}]
      \node[circle,draw, fill=black!100,thick] (j) at (0,0) [label=right:$j$] {};
  \end{tikzpicture} & $ f^J $ & $\sum b_j$ & $=$ & $1$ \\
\hline
  $b$ &
  \begin{tikzpicture}[scale=.5]
      [meagre/.style={circle,draw, fill=black!100,thick},
      fat/.style={circle,draw,thick}]
      \node[circle,draw, fill=black!100,thick] (j) at (0,0) [label=right:$j$] {};
      \node[circle,draw, fill=black!100,thick] (k) at (1,1) [label=right:$k$] {};
      \draw[-] (j) -- (k);
  \end{tikzpicture} & $ f^J_Kf^K $ & $\sum b_j \beta_{j,k}$ & $=$ & $1/2$ \\
\hline
  $c$ &
  \begin{tikzpicture}[scale=.5]
      [meagre/.style={circle,draw, fill=black!100,thick},
      fat/.style={circle,draw,thick}]
      \node[circle,draw, fill=black!100,thick] (j) at (1,0) [label=right:$j$] {};
      \node[circle,draw, fill=black!100,thick] (k) at (2,1) [label=right:$k$] {};
      \node[circle,draw, fill=black!100,thick] (l) at (0,1) [label=right:$l$] {};
      \draw[-] (j) -- (k);
      \draw[-] (j) -- (l);
  \end{tikzpicture} & $ f^J_{KL}f^Kf^L $ & $\sum b_j \alpha_{j,k}\alpha_{j,l}$ & $=$ & $1/3$ \\
\hline
  $d$ &
  \begin{tikzpicture}[scale=.5]
      [meagre/.style={circle,draw, fill=black!100,thick},
      fat/.style={circle,draw,thick}]
      \node[circle,draw, fill=black!100,thick] (j) at (0,0) [label=right:$j$] {};
      \node[circle,draw, fill=black!100,thick] (k) at (1,1) [label=right:$k$] {};
      \node[circle,draw, fill=black!100,thick] (l) at (0,2) [label=right:$l$] {};
      \draw[-] (j) -- (k);
      \draw[-] (k) -- (l);
  \end{tikzpicture} & $ f^J_Kf^K_Lf^L $ & $\sum b_j \beta_{j,k}\beta_{k,l}$ & $=$ & $1/6$ \\
\hline
  $e$ &
  \begin{tikzpicture}[scale=.5]
      [meagre/.style={circle,draw, fill=black!100,thick},
      fat/.style={circle,draw,thick}]
      \node[circle,draw, fill=black!100,thick] (j) at (1,0) [label=right:$j$] {};
      \node[circle,draw, fill=black!100,thick] (k) at (2,1) [label=right:$k$] {};
      \node[circle,draw, fill=black!100,thick] (l) at (1,1) [label=above:$l$] {};
      \node[circle,draw, fill=black!100,thick] (m) at (0,1) [label=left:$m$] {};
      \draw[-] (j) -- (k);
      \draw[-] (j) -- (l);
      \draw[-] (j) -- (m);
  \end{tikzpicture} & $ f^J_{KLM}f^Kf^Lf^M $ & $\sum b_j \alpha_{j,k}\alpha_{j,l}\alpha_{jm}$ & $=$ & $1/4$ \\
\hline
  $f$ &
  \begin{tikzpicture}[scale=.5]
      [meagre/.style={circle,draw, fill=black!100,thick},
      fat/.style={circle,draw,thick}]
      \node[circle,draw, fill=black!100,thick] (j) at (1,0) [label=right:$j$] {};
      \node[circle,draw, fill=black!100,thick] (k) at (2,1) [label=right:$k$] {};
      \node[circle,draw, fill=black!100,thick] (l) at (1,2) [label=right:$l$] {};
      \node[circle,draw, fill=black!100,thick] (m) at (0,1) [label=left:$m$] {};
      \draw[-] (j) -- (k);
      \draw[-] (k) -- (l);
      \draw[-] (j) -- (m);
  \end{tikzpicture} & $ f^J_{KM}f^K_Lf^Lf^M $ & $\sum b_j \alpha_{j,k}\beta_{k,l}\alpha_{j,m}$ & $=$ & $1/8$ \\
\hline
  $g_1$ &
  \begin{tikzpicture}[scale=.5]
      [meagre/.style={circle,draw, fill=black!100,thick},
      fat/.style={circle,draw,thick}]
      \node[circle,draw, fill=black!100,thick] (j) at (0,0) [label=right:$j$] {};
      \node[circle,draw, fill=black!100,thick] (k) at (1,1) [label=right:$k$] {};
      \node[circle,draw, fill=black!100,thick] (l) at (0,2) [label=above:$l$] {};
      \node[circle,draw, fill=black!100,thick] (m) at (2,2) [label=above:$m$] {};
      \draw[-] (j) -- (k);
      \draw[-] (k) -- (l);
      \draw[-] (k) -- (m);
  \end{tikzpicture} & $ f^J_Kf^K_{LM}f^Lf^M $ & $\sum b_j \alpha_{j,k}\alpha_{k,m}\alpha_{k,l}$ & $=$ & $1/12$ \\
  $g_2$ &
  \begin{tikzpicture}[scale=.5]
      [meagre/.style={circle,draw, fill=black!100,thick},
      fat/.style={circle,draw,thick}]
      \node[circle,draw, thick] (j) at (0,0) [label=right:$j$] {};
      \node[circle,draw, fill=black!100,thick] (k) at (1,1) [label=right:$k$] {};
      \node[circle,draw, fill=black!100,thick] (l) at (0,2) [label=above:$l$] {};
      \node[circle,draw, fill=black!100,thick] (m) at (2,2) [label=above:$m$] {};
      \draw[-] (j) -- (k);
      \draw[-] (k) -- (l);
      \draw[-] (k) -- (m);
  \end{tikzpicture} & $ \mathbf{A}_{JK}f^K_{LM}f^Lf^M $ & $\sum b_j \gamma_{j,k}\alpha_{k,m}\alpha_{k,l}$ & $=$ & $0$ \\
\hline
  $h$ &
  \begin{tikzpicture}[scale=.5]
      [meagre/.style={circle,draw, fill=black!100,thick},
      fat/.style={circle,draw,thick}]
      \node[circle,draw,fill=black!100, thick] (j) at (0,0) [label=left:$j$] {};
      \node[circle,draw, fill=black!100,thick] (k) at (1,1) [label=right:$k$] {};
      \node[circle,draw, fill=black!100,thick] (l) at (0,2) [label=left:$l$] {};
      \node[circle,draw, fill=black!100,thick] (m) at (1,3) [label=right:$m$] {};
      \draw[-] (j) -- (k);
      \draw[-] (k) -- (l);
      \draw[-] (l) -- (m);
  \end{tikzpicture} & $ f^J_Kf^K_Lf^L_Mf^M $ & $\sum b_j \beta_{j,k}\beta_{k,l}\beta_{l,m}$ & $=$ & $1/24$ \\
\hline
  \end{tabular}
  \end{center}
\end{figure}
%%%%%%%%%%%%%%%%%%%%%%%%

%%%%%%%%%%%%%%%%%%%%%%%%
\begin{figure}[htp]
  \label{fig:extraROK}
  \caption{Additional TK-trees and Rosenbrock-Krylov conditions for order five.}
  \begin{center}
  \begin{tabular}{|c|c|lcr|}
\hline
  \begin{tikzpicture}[scale=.5]
      [meagre/.style={circle,draw, fill=black!100,thick},
      fat/.style={circle,draw,thick}]
      \node[circle,draw,thick] (j) at (0,0) [label=right:$j$] {};
      \node[circle,draw, fill=black!100,thick] (k) at (1,1) [label=right:$k$] {};
      \node[circle,draw, fill=black!100,thick] (l) at (2,2) [label=right:$l$] {};
      \node[circle,draw, fill=black!100,thick] (m) at (1,2) [label=above:$m$] {};
      \node[circle,draw, fill=black!100,thick] (p) at (0,2) [label=left:$p$] {};
      \draw[-] (j) -- (k);
      \draw[-] (k) -- (l);
      \draw[-] (k) -- (m);
      \draw[-] (k) -- (p);
  \end{tikzpicture} & $ \mathbf{A}_{JK}f^K_{LMP}f^Lf^Mf^P $ & $\sum \gamma_{j,k} \alpha_{k,l} \alpha_{k,m} \alpha_{k,p} $ & $=$ & $0$ \\
\hline
   \begin{tikzpicture}[scale=.5]
      [meagre/.style={circle,draw, fill=black!100,thick},
      fat/.style={circle,draw,thick}]
      \node[circle,draw, thick] (j) at (0,0) [label=right:$j$] {};
      \node[circle,draw, fill=black!100,thick] (k) at (1,1) [label=right:$k$] {};
      \node[circle,draw, fill=black!100,thick] (l) at (0,2) [label=left:$l$] {};
      \node[circle,draw, fill=black!100,thick] (m) at (1,3) [label=above:$m$] {};
      \node[circle,draw, fill=black!100,thick] (p) at (2,2) [label=right:$p$] {};
      \draw[-] (j) -- (k);
      \draw[-] (k) -- (l);
      \draw[-] (k) -- (p);
      \draw[-] (l) -- (m);
  \end{tikzpicture} & $ \mathbf{A}_{JK}f^K_{LP}f^L_Mf^Mf^P $ & $\sum \gamma_{j,k}\alpha_{k,p} \alpha_{k,l} \alpha_{l,m}$ & $=$ & $0$ \\
\hline
   \begin{tikzpicture}[scale=.5]
      [meagre/.style={circle,draw, fill=black!100,thick},
      fat/.style={circle,draw,thick}]
      \node[circle,draw, thick] (j) at (1,0) [label=right:$j$] {};
      \node[circle,draw, fill=black!100,thick] (k) at (2,1) [label=right:$k$] {};
      \node[circle,draw, fill=black!100,thick] (l) at (1,2) [label=left:$l$] {};
      \node[circle,draw, fill=black!100,thick] (m) at (2,3) [label=above:$m$] {};
      \node[circle,draw, fill=black!100,thick] (p) at (0,3) [label=above:$p$] {};
      \draw[-] (j) -- (k);
      \draw[-] (k) -- (l);
      \draw[-] (l) -- (p);
      \draw[-] (l) -- (m);
  \end{tikzpicture} & $ \mathbf{A}_{JK}f^K_Lf^L_{MP}f^Pf^M $ & $\sum \gamma_{j,k} \alpha_{k,l} \alpha_{l,m} \alpha_{l,p}$ & $=$ & $0$ \\
\hline
   \begin{tikzpicture}[scale=.5]
      [meagre/.style={circle,draw, fill=black!100,thick},
      fat/.style={circle,draw,thick}]
      \node[circle,draw, thick] (j) at (1,0) [label=right:$j$] {};
      \node[circle,draw, thick] (k) at (2,1) [label=right:$k$] {};
      \node[circle,draw, fill=black!100,thick] (l) at (1,2) [label=left:$l$] {};
      \node[circle,draw, fill=black!100,thick] (m) at (2,3) [label=above:$m$] {};
      \node[circle,draw, fill=black!100,thick] (p) at (0,3) [label=above:$p$] {};
      \draw[-] (j) -- (k);
      \draw[-] (k) -- (l);
      \draw[-] (l) -- (p);
      \draw[-] (l) -- (m);
  \end{tikzpicture} & $ \mathbf{A}_{JK}\mathbf{A}_{KL}f^L_{MP}f^Pf^M $ & $\sum \gamma_{j,k} \gamma_{k,l} \alpha_{l,m} \alpha_{l,p}$ & $=$ & $0$ \\
\hline
   \begin{tikzpicture}[scale=.5]
      [meagre/.style={circle,draw, fill=black!100,thick},
      fat/.style={circle,draw,thick}]
      \node[circle,draw, fill=black!100, thick] (j) at (1,0) [label=right:$j$] {};
      \node[circle,draw,thick] (k) at (2,1) [label=right:$k$] {};
      \node[circle,draw, fill=black!100,thick] (l) at (1,2) [label=left:$l$] {};
      \node[circle,draw, fill=black!100,thick] (m) at (2,3) [label=above:$m$] {};
      \node[circle,draw, fill=black!100,thick] (p) at (0,3) [label=above:$p$] {};
      \draw[-] (j) -- (k);
      \draw[-] (k) -- (l);
      \draw[-] (l) -- (p);
      \draw[-] (l) -- (m);
  \end{tikzpicture} & $ f^J_K\mathbf{A}_{KL}f^L_{MP}f^Pf^M $ & $\sum \alpha_{j,k} \gamma_{k,l} \alpha_{l,m} \alpha_{l,p}$ & $=$ & $0$ \\
\hline
  \end{tabular}
  \end{center}
\end{figure}

%%%%%%%%%%%%%%%%%%%%%%%%
\subsection{Rosenbrock-K methods of type 2}
%%%%%%%%%%%%%%%%%%%%%%%%

\ \\

\begin{definition}
A Rosenbrock-K method of type 2 is given by Algorithm \ref{ROK-nonautonomous-step} and uses an enriched underlying Krylov subspace, where additional basis vectors are added to those in \eqref{eqn:krylovspace}. The additional basis vectors are chosen such that different elementary differentials associated with trees in $TK\backslash T$ are equal to those of similar trees in $T$. Consequently, the order conditions of a type 2 Rosenbrock-K method are the same as those of classical Rosenbrock methods.
\end{definition}

For example, consider the tree $g_2$ in Table \ref{fig:ROKtrees}. The corresponding term in the Taylor series of the solution is
$\mathbf{A}_n \cdot f_{y,y}(f_n,f_n)$\,.
The application of the second derivative tensor $f_{y,y}$  to a pair of function values results in the vector
\begin{equation}
\label{eqn:second-derivative}
u_{g_2}^k = \left( f_{y,y}(f_n,f_n) \right)^k = \sum_{\ell,m=1}^N \left. \frac{\partial^2 f^k}{\partial y^\ell\, \partial y^m}\right|_{y=y_n}\, f^\ell(y_n)\, f^m(y_n)\,, \quad
k=1,\dots,N\,.
\end{equation}
To obtain a type 2 Rosenbrock-K method of order four the Krylov space \eqref{eqn:krylovspace} is extended as follows:
\begin{eqnarray}
\label{eqn:krylovspace-extended}
	\K_{M+2}(\mathbf{J}_n\,,f_n) &=& \textnormal{span}\left\{\, f_n, \, \mathbf{J}_n\,f_n, \, \dots, \mathbf{J}_n^{M-1}\,f_n\,, u_{g_2}\,, \mathbf{J}_n\, u_{g_2}  \, \right\} \\
\nonumber
	&=& \textnormal{span}\left\{v_1, v_2,  \dots, v_{M+2}\right\} \,.
\end{eqnarray}
The Jacobian approximation is $\mathbf{A}_n = \mathbf{V}_{n;M}\, \mathbf{H}_{n;M}\, \mathbf{V}_{n;M}^T$ where $\mathbf{V}_{n;M} \in \R^{N \times (M+2)}$, and $\mathbf{H}_{n;M} = \mathbf{V}_{n;M}^T\, \mathbf{J}_n\, \mathbf{V}_{n;M} \in \R^{(M+2) \times (M+2)}$ is no longer upper Hessenberg. 

The construction \eqref{eqn:krylovspace-extended} ensures that the elementary differential of the tree $g_2 \in TW \backslash T$ coincides with the
elementary differential of a regular Butcher tree:
\[
\mathbf{A}_n \, f_{y,y}(f_n,f_n) = \mathbf{A}_n \, u_{g_2} = \mathbf{V}_{n;M}\, \mathbf{V}_{n;M}^T \, \mathbf{J}_n \underbrace{\mathbf{V}_{n;M}\, \mathbf{V}_{n;M}^T \,  u_{g_2}}_{=u_{g_2}}
= \mathbf{V}_{n;M}\, \mathbf{V}_{n;M}^T \, \mathbf{J}_n\,  u_{g_2}  = \mathbf{J}_n\,  u_{g_2} = \mathbf{J}_n\, f_{y,y}(f_n,f_n) \,.
\]
We have the following interesting consequences.
\begin{remark}
Any classical Rosenbrock method of order four (or higher) becomes a
type 2 Rosenbrock-K method of order four when the Jacobian approximation \eqref{eqn:ROK-approximation} uses a  six-dimensional extended Krylov space (more exactly, when the underlying Krylov space is \eqref{eqn:krylovspace-extended} with $M \ge 4$).
\end{remark}
\begin{remark}
General Rosenbrock-K methods of any order can be obtained by combining the type 2 and type 1 approaches.
Specifically, some of the trees in $TK\backslash T$ are recolored (to obtain the similar trees in $T$) by extending the underlying Krylov space, i.e., by using a type 2 approach. The elementary differentials corresponding to the remaining trees in $TK\backslash T$ are then cancelled by imposing additional type 1 order conditions.
\end{remark}

%%%%%%%%%%%%%%%%%%%%%%%%
\subsection{Implementation aspects} \label{sec:implementation}
%%%%%%%%%%%%%%%%%%%%%%%%

%
\iffalse
{\mt The discussion at the end of page 13 concerning the errors introduced by approximating the
Jacobian with finite-difference approximations is of great relevance to PDE applications.
Could the authors provide any comments about when the error bound
in the last equation on page 13 (and the assumption that the Jacobian powers are uniformly
bounded) are reasonable? In other words, does this result present a practical bound on the
attainable order from these methods?}
\fi
%

The cost of Rosenbrock-K integration for large-scale systems is dominated by the cost of building the Krylov space at each step.
The Arnoldi iteration requires $M$ Jacobian-vector products, as well as vector operations totaling $\mathcal{O}(M^2 N)$ operations during orthogonalization \cite{Saad}.  
Both Jacobian-vector products and vector operations can be efficiently parallelized.

Jacobian-vector can be obtained in several ways.  Most straightforwardly the entire Jacobian matrix can be constructed and then  a Jacobian-vector product can be calculated in the usual way.  This process is expensive both in terms of storage and computation.

An alternative is to implement a routine that computes directly Jacobian-vector products without building the Jacobian matrix. Such a routine can be obtained through forward-mode automatic differentiation \cite{Griewank2000EDP} and its cost is similar to the ODE function computations.  
Large distributed applications rely on an infrastructure which partitions the solution vector $y_n$ across
nodes.  The computation of the ODE function $f_n$ is done in parallel. 
Data exchange of data among subdomains is needed
in order to fulfill grid dependencies.  
Exact Jacobian-vector products $J_n \cdot u$ are obtained by linearizing the ODE function 
$f_n$ in the direction $u$. Therefore Jacobian-vector operations can be computed element by element,
inherit the parallel structure of the ODE function calculation (e.g., 
parallelism obtained by domain decomposition), and can be implemented very efficiently using the same 
parallel software infrastructure. 
The same data partitioning can be used for both the solution $y_n$ and the vector $u$. Note that successive Arnoldi iterations 
act on the distributed vectors $u$ without any need for global communication (only local communication of the 
boundary elements is needed at each iteration).

Jacobian-vector products can also be approximated by finite differences of the form \cite{Keyes_2004_JFNK}
%%%%%%%%%%%%%%%%%%%%%%%%
\begin{equation}
\label{eqn:jacvec}
\mathbf{J}_n\, u \approx \frac{f(t_n,y_n + \delta\, u) - f(t_n,y_n)}{\delta}\,.
\end{equation}
%%%%%%%%%%%%%%%%%%%%%%%%
The increment $\delta$ is related to machine precision. Equation \eqref{eqn:jacvec} is sometimes referred to as a ``matrix-free''
approximation. For example, ``Jacobian-free Newton-Krylov'' methods \cite{Keyes_2004_JFNK} employ the approximation  
\eqref{eqn:jacvec} within linear Krylov space solvers in the context of Newton iterations for nonlinear systems.
Clearly the finite difference approximation  \eqref{eqn:jacvec} uses the same data partitioning for 
$y_n$ and $u$, and inherits the parallel performance of the ODE function calculation.

Finite differences can also be used to approximate higher order derivatives.
For example, the second derivative term \eqref{eqn:second-derivative} can approximated by finite differences of Jacobian-vector products, as follows:
\begin{eqnarray*}
f_{y,y}(u,u) &\approx&  \frac{f_y(t_n,y_n + \delta\, u)\cdot u - \mathbf{J}_n\cdot u}{\delta},
%\approx  \frac{f(t_n,y_n + 2\delta\, u) - 2\,f(t_n,y_n + \delta\, u) + f(t_n,y_n)}{\delta^2}\,.
\end{eqnarray*}
and where each Jacobian-vector product can also be approximated,

%%%%%%%%%%%%%%%%%%%%%%%%
\subsection{Errors due to finite difference approximations} \label{sec:finite-diff}
%%%%%%%%%%%%%%%%%%%%%%%%

An analysis of matrix-free Newton-Krylov methods is provided in \cite[Theorem 2.3]{Brown_2008_JFNK}. Assume that the Arnoldi process with
the exact Jacobian-vector products produces $\mathbf{H}_{n;M}$, $\mathbf{V}_{n;M}$, while the Arnoldi process using finite difference approximations \eqref{eqn:jacvec} produces $\tilde{\mathbf{H}}_{n;M}$ and $\tilde{\mathbf{V}}_{n;M}$. The errors in the finite difference approximations \eqref{eqn:jacvec} during each Arnoldi iteration are
\begin{eqnarray*}
%e_0 &=& q_0-J_n\cdot x_0, \quad q_0=(F(\cdot+\delta_0 x_0)-F(\cdot))/\delta \\
e_i &=&  \frac{f(t_n, y_n + \delta_i\, \tilde{v}_i) - f(t_n,y_n)}{\delta_i} - \mathbf{J}_n\cdot \tilde{v}_i, \quad i=1,\dots,M-1\,.
\end{eqnarray*}
Collect these error vectors, together with $e_0=0$ (the error in computing $\tilde{v}_1=f_n/\Vert f_n \Vert$), in the matrix 
\[
\mathbf{E} = [e_0,e_1,\dots,e_{M-1}] \in \R^{N \times M}\,.
\]
According to \cite[Theorem 2.3]{Brown_2008_JFNK}, the matrices $\tilde{\mathbf{H}}_{n;M}$ and $\tilde{\mathbf{V}}_{n;M}$
can be obtained by an application of the {\it exact} Arnoldi process (i.e., with exact Jacobian-vector products) to obtain a basis of
the modified space
\[
\mathcal{K}_M\left( \tilde{\mathbf{J}}_n, f_n \right)\, \quad \textnormal{with} \quad \tilde{\mathbf{J}}_n := \mathbf{J}_n+\mathbf{E}\, \tilde{\mathbf{V}}_{n;M}^T\,.
\]
According to Lemma \ref{lemma:ROK-matrix}, the matrix approximation $\tilde{\mathbf{A}_n}= \tilde{\mathbf{V}}_{n;M} \tilde{\mathbf{H}}_{n;M} \tilde{\mathbf{V}}_{n;M}^T$ has the following property
for any $0 \le k \le M-1$
\begin{eqnarray*}
\tilde{\mathbf{A}}_n^k\, f_n &=& \tilde{\mathbf{J}}_n^k\, f_n\\
&=&\left( \mathbf{J}_n+\mathbf{E}\, \tilde{\mathbf{V}}_{n;M}^T\right)^k \, f_n\\
%&=&  \mathbf{J}_n^k \, f_n + k\, \mathbf{J}_n^{k-1}\, \mathbf{E}\, \tilde{\mathbf{V}}_{n;M}^T \, f_n + \dots \\
&=&  \mathbf{J}_n^k \, f_n + \sum_{i=0}^{k-1} \begin{pmatrix} k \\ i \end{pmatrix}\, \mathbf{J}_n^{i}\, (\mathbf{E}\, \tilde{\mathbf{V}}_{n;M}^T)^{k-i} \,  f_n  \\
&=&  \mathbf{J}_n^k \, f_n + \mathcal{O}(\Vert \mathbf{E} \Vert)\,,
\end{eqnarray*}
where for the last equality we have made the assumption that the Jacobian powers are uniformly bounded.

When exact Jacobian-vector products are used no additional type 1 Rosenbrock-K order conditions are imposed for trees in $TW\backslash TK$.
Similarly, when higher derivatives are computed exactly no additional order conditions are needed for type 2 methods.
When finite difference approximations are used, however, the elementary differentials of trees in $TW\backslash TK$ appear in the
expansion of the numerical solution with nonzero coefficients of size $\mathcal{O}(\Vert \mathbf{E} \Vert)$, i.e., of the size of the
{\it absolute} errors incurred in the finite difference approximations. If the finite difference approximations are not sufficiently accurate 
the order of the Rosenbrock-K methods may be lost.

For example, consider the tree $d_2 \in TW \backslash TK$ in Figure \ref{fig:ROWtrees}.
When finite differences are used, it contributes the following
$\mathcal{O}\left(\Vert \mathbf{E} \Vert\, h^2\right)$  term to the local error 
\begin{equation}
\label{eqn:ROW-2}
h^2\, \left(\sum_{j=1}^s b_j \gamma_{j,k}\right) (\mathbf{E}\, \tilde{v}_1) \, \Vert f_n \Vert\,.
\end{equation}
In order to ensure that the Rosenbrock-K method preserves its order $p$, a sufficient condition is that the finite difference errors are bounded by
\[
\Vert \mathbf{E} \Vert \le C\, h^{p-1}\,.
\]
Without assuming the uniform boundedness of the Jacobian a sufficient condition is $\Vert \mathbf{J}_n^{i}\, (\mathbf{E}\, \tilde{\mathbf{V}}_{n;M}^T)^{k-i} \Vert \le C\, h^{p-1}$, $i=0,\dots,k-1$.
When the exact Jacobian-vector products are unavailable, and when the finite difference approximations cannot be computed this accurately, it may be advantageous to  choose Rosenbrock-K methods whose coefficients satisfy the full set of Rosenbrock-W conditions.

% More complete explanation follows:
%The finite difference error are $e_0=q_0-J_n\cdot x_0$, $q_0=(F(\cdot+\delta_0 x_0)-F(\cdot))/\delta$
%and $e_i=q_i-J_n\cdot \tilde{v}_i$, $q_i=(f(y_n+\delta_i\tilde{v}_i)-f(y_n))/\delta_i$. Then $\tilde{H},\tilde{V}$ are the same as
%the application of exact Arnoldi
%to the modified system $\tilde{J}_n := J_n+E\, \tilde{V}^T$ and $\tilde{f}_n := f_n-e_0+E\, \tilde{V}^T\, x_0$.
%Note that for the initial guess $x_0=0$ we have and $e_0$=0 and the RHS $f$ is the same. Only the Jacobian matrix changes.

%%%%%%%%%%%%%%%%%%%%%%%%
\section{Construction of Rosenbrock-K methods of order four} \label{sec:solvingorderconds}
%%%%%%%%%%%%%%%%%%%%%%%%
We now construct practical type 1 Rosenbrock-K methods of order four. We consider the case with $\gamma_{i,i} = \gamma$ for all $i$, and denote 
\[
	\beta_i = \displaystyle\sum_{j=1}^{i} \beta_{i,j} = \alpha_i + \gamma_i \,, \quad
	\beta_i' = \displaystyle\sum_{j=1}^{i-1} \beta_{i,j} \,.
\]
We examine numerically the linear stability properties of the resulting methods when using the exact Jacobian so that $\mathbf{A}_n = \mathbf{J}_n$.  Rosenbrock-K methods share the same stability function with classical Rosenbrock methods 
\begin{equation}
\label{eqn:stabilityfunc}
R(z) = 1 + zb^T\left(\mathbf{I}_{s \times s} - z\beta\right)^{-1}\mathds{1}\,,
\end{equation}
where $\mathds{1} \in \R^s$ is a vector of ones.

%%%%%%%%%%%%%%%%%%%%%%%%
\subsection{ROK4a: a four stages, fourth order, L-stable Rosenbrock-K method}
%%%%%%%%%%%%%%%%%%%%%%%%
We start with constructing a four stages,  fourth order Rosenbrock-K method.
The order conditions are as follows:
\begin{equation}
\label{eqn:ocfours}
\renewcommand{\arraystretch}{1.5}
\begin{array}{clcl}
	(a) & b_1 + b_2 + b_3 + b_4 & = & 1 \\
	(b) & b_2 \beta_2' + b_3 \beta_3' + b_4 \beta_4' & = & \frac{1}{2} - \gamma = p_{21}(\gamma)\\
	(c) & b_2 \alpha_2^2 + b_3 \alpha_3^2 + b_4 \alpha_4^2 & = &  \frac{1}{3} \\
	(d) & b_3(\beta_{3,2} \beta_2') + b_4(\beta_{4,2}\beta_2' + \beta_{4,3}\beta_3') & = & \frac{1}{6} - \gamma + \gamma^2 = p_{3,2}(\gamma) \\
	(e) & b_2 \alpha_2^3 + b_3\alpha_3^3 + b_4\alpha_4^3 & = & \frac{1}{4} \\
	(f) & b_3\alpha_3\alpha_{3,2}\beta_2' + b_4\alpha_4(\alpha_{4,2}\beta_2' + \alpha_{4,3}\beta_3) & = & \frac{1}{8} - \frac{1}{3} \gamma = p_{4,2}(\gamma) \\
	(g_1) & b_3\alpha_{3,2}\alpha_2^2 + b_4(\alpha_{4,2}\alpha_2^2 + \alpha_{4,3}\alpha_3^2) & = & \frac{1}{12} \\
	(g_2) & b_3\gamma_{3,2}\alpha_2^2 + b_4(\gamma_{4,2}\alpha_2^2 + \gamma_{4,3}\alpha_3^2) & = & -\frac{1}{3}\gamma \\
	(h) & b_4\beta_{4,3}\beta_{3,2}\beta_{2}' & = & \frac{1}{24} - \frac{1}{2}\gamma + \frac{3}{2}\gamma^2 - \gamma^3 = p_{4,4}(\gamma)
\end{array}
\end{equation}
To solve \eqref{eqn:ocfours} we follow the solution process outlined in \cite[Section IV.7]{Hairer_book_II}.  First we choose $\gamma = 0.572816062482135$ so that $R(\infty) = 0$, where $R(z)$ is the stability function of the method.  We then treat equations (\ref{eqn:ocfours}.a), (\ref{eqn:ocfours}.c), and (\ref{eqn:ocfours}.e) separately, as a linear system in $b_i$'s.  We make the arbitrary choices
\[ b_3 = 0, \quad \alpha_2 = \frac{1}{2}, \ \  \alpha_3 = 1, \quad \alpha_4 = 1\,, \]
and the solve the system
\begin{equation*}
\label{eqn:step24s}
	\left[\begin{array}{ccc} 1 & 1 & 1 \\ 0 & \alpha_2^2 & \alpha_4^2 \\ 0 & \alpha_2^3 & \alpha_4^3 \end{array} \right]  \left[ \begin{array}{c} b_1 \\ b_2 \\ b_4 \end{array} \right] = \left[ \begin{array}{c} 1 \\ \frac{1}{3} \\ \frac{1}{4} \end{array} \right]
\end{equation*}
to obtain $b_1, b_2$, and $b_4$.

In order to allow for the existence of an embedded method of order three we require that the third order conditions are not satisfied uniquely.  The following equation guarantees that by setting the determinant of the  system of third order conditions to zero:
\begin{equation}
\label{eqn:embedcond}
	(\beta_2' \alpha_4^2 - \beta_4'\alpha_2)\, \beta_{3,2}\, \beta_2' = (\beta_2' \alpha_3^2 - \beta_3' \alpha_2^2)\displaystyle\sum_{j=2}^3 \beta_{4,j}\beta_j'\,.
\end{equation}

We now take $\beta_{4,3}$ as a free parameter and compute $\beta_{3,2}\beta_2'$ from (\ref{eqn:ocfours}.h) and $(\beta_{4,2}\beta_2' + \beta_{4,3}\beta_3')$ from (\ref{eqn:ocfours}.d).  Inserting these expressions into (\ref{eqn:embedcond}) yields a relation between $\beta_2'$, $\beta_3'$, $\beta_4'$.  Eliminating $(b_4\beta_{4,2} + b_3\beta_{3,2})$ from (\ref{eqn:ocfours}.d), and from \{(\ref{eqn:ocfours}.$g_1$) + (\ref{eqn:ocfours}.$g_2$)\}, yields a second relation.  A third relation is obtained from (\ref{eqn:ocfours}.b), and this leads to the following system for $\beta_2'$, $\beta_3'$, $\beta_4'$:
\begin{equation*}
\label{eqn:step34s}
\left[\begin{array}{ccc} \alpha_4^2 \frac{p_{4,4}}{b_4\beta_{4,3}} - \alpha_3^2 \frac{p_{3,2}}{b_4} & \alpha_2^2 \frac{p_{3,2}}{b_4} & -\alpha_2^2\frac{p_{4,4}}{b_4\beta_{4,3}} \\ b_2 & b_3 & b_4 \\ b_4\beta_{4,3}\alpha_3^2 - p_{4,3} & -b_4\beta_{4,3}\alpha_2^2 & 0 \end{array}\right] \left[ \begin{array}{c} \beta_2' \\ \beta_3' \\ \beta_4' \end{array}\right] = \left[  \begin{array}{c} 0 \\ p_{2,1} \\ -\alpha_2^2 p_{3,2} \end{array} \right]\,.
\end{equation*}
Here we make the arbitrary choice
\[ \beta_{43} = -\frac{1}{4} \]
and compute $\beta_{3,2}$ and $\beta_{4,2}$ from 
\begin{equation*}
\label{eqn:step44s}
	\beta_{3,2} = \frac{p_{4,4}}{b_4\beta_{4,3}\beta_2'}, \quad  \beta_{4,2} = \frac{p_{3,2} - b_4\beta_{4,3}\beta_3'}{b_4\beta_2'}\,.
\end{equation*}
Next we impose directly equations (\ref{eqn:ocfours}.f), (\ref{eqn:ocfours}.$g_1$), and (\ref{eqn:ocfours}.$g_2$) along with the definition of $\beta_{i,j}$:
\begin{equation*}
\label{eqn:step54s}
	\begin{array}{lcl}
	 b_3\alpha_{3,2}\beta_2' + b_4(\alpha_{4,2}\beta_2' + \alpha_{4,3}\beta_3) & = & p_{4,2} \\
	 b_3\alpha_{3,2}\alpha_2^2 + b_4(\alpha_{4,2}\alpha_2^2 + \alpha_{4,3}\alpha_3^2) & = & \frac{1}{12} \\
	 b_3\gamma_{3,2}\alpha_2^2 + b_4(\gamma_{4,2}\alpha_2^2 + \gamma_{4,3}\alpha_3^2) & = & -\frac{1}{3}\gamma \\
	\gamma_{3,2} + \alpha_{3,2} & = & \beta_{3,2} \\
	\gamma_{4,2} + \alpha_{4,2} & = & \beta_{4,2} \\
	\gamma_{4,3} + \alpha_{4,3} & = & \beta_{4,3} \\
	\end{array}
\end{equation*}
Finally $\alpha_{i,1}$ and $\beta_{i,1}$ follow immediately from the definition of $\alpha_i$ and $\beta_i'$ respectively. The coefficient values for this method, named ROK4a, are given in Table \ref{table:ROK4a-coef}.

\begin{table}
\begin{center}
  \begin{tabular}{|rcrrcr|}
\hline
\multicolumn{6}{|l|}{   $\gamma$  =   0.572816062482135 } \\
\hline
   $\alpha_{2,1}$ & = & 1 		  				& $\gamma_{2,1}$ & = &-1.91153192976055097824 \\
   $\alpha_{3,1}$ & = & 0.10845300169319391758 	& $\gamma_{3,1}$ & = & 0.32881824061153522156 \\
   $\alpha_{3,2}$ & = & 0.39154699830680608241 	& $\gamma_{3,2}$ & = & 0.0 \\
   $\alpha_{4,1}$ & = & 0.43453047756004477624	& $\gamma_{4,1}$ & = & 0.03303644239795811290 \\
   $\alpha_{4,2}$ & = & 0.14484349252001492541	& $\gamma_{4,2}$ & = &-0.24375152376108235312 \\
   $\alpha_{4,3}$ & = &-0.07937397008005970166 	& $\gamma_{4,3}$ & = &-0.17062602991994029834 \\
\hline
   $b_1$ 	 & = & 0.16666666666666666667 	& $\widehat{b}_1$   & = & 0.50269322573684235345 \\
   $b_2$	 & = & 0.16666666666666666667 	& $\widehat{b}_2$   & = & 0.27867551969005856226 \\
   $b_3$	 & = & 0.0		  			& $\widehat{b}_3$   & = & 0.21863125457309908428 \\
   $b_4$	 & = & 0.66666666666666666667 	& $\widehat{b}_4$   & = & 0.0		     \\
\hline
  \end{tabular}
\caption{Coefficients of ROK4a, a fourth order, L-stable, type 1 Rosenbrock-K method.}
\label{table:ROK4a-coef}
\end{center}
\end{table}
The choice of $\gamma$ ensures that for the main method  $R(\infty) = 0$.  The embedded method has $\widehat{R}(\infty) = -0.55$ Figure \ref{fig:stability-ROK4} shows the stability function values for both the main and embedded methods along the imaginary axis. We see that the absolute function values are below one, which implies that the main {\sc Rok4}a method is L-stable, and the embedded method is strongly A-stable (i.e., $|\widehat{R}(\infty)| < 1$).

It is important to note that the stability results presented here apply to the case where a full Jacobian is used, and does not account for the impact of Krylov approximation.  
The impact of the Krylov approximation on the stability will be the subject of future work.

Exact stability requirements when making use of the Krylov approximation of the Jacobian are as yet undetermined, though a result by Wensch in \cite{Wensch_2005_DAE} gives reason to believe that the size of the Krylov subspace must be as large as the number of stiff variables in the underlying problem.

%%%%%%%%%%%%%%%%%%%%%%%%
\subsection{ROK4b: a six stages,  fourth order stiffly accurate Rosenbrock-K method}\label{sec:Rok4b}
%%%%%%%%%%%%%%%%%%%%%%%%

Stiff accuracy is a desirable property when solving very stiff systems or index-1 differential algebraic equations. A stiffly accurate Rosenbrock method \cite[Section IV.4]{Hairer_book_II} is characterized by the property 
\[
b_i = \beta_{s,i}\,, \quad i=1,\dots, s\,.
\]
We have derived a six-stage, stiffly accurate, fourth-order Rosenbrock-K method, named ROK4b.   For brevity we do not show here the order conditions, nor we present the solution method. The coefficients have been obtained through a process similar to that outlined in Section \ref{sec:Rok4p}.  The {\sc Rok4}b method coefficients are shown in Table \ref{table:ROK4b-coef}.
\begin{table}
\begin{center}
  \begin{tabular}{|rcrrcr|}
\hline
\multicolumn{6}{|l|}{   $\gamma$  =  0.31 } \\
\hline
   $\alpha_{2,1}$ & = &  1.0 		 			& $\gamma_{2,1}$ & = & -22.824608269858540		   \\
   $\alpha_{3,1}$ & = &  0.530633333333333 		& $\gamma_{3,1}$ & = & -69.343635255712726 \\
   $\alpha_{3,2}$ & = & -0.030633333333333		& $\gamma_{3,2}$ & = & -0.030633333333333 \\
   $\alpha_{4,1}$ & = &  0.894444444444444 		& $\gamma_{4,1}$ & = &  404.7106882480958  \\
   $\alpha_{4,2}$ & = &  0.055555555555556		& $\gamma_{4,2}$ & = &  0.055555555555556		   \\
   $\alpha_{4,3}$ & = &  0.05			  		& $\gamma_{4,3}$ & = &  0.05  \\ 
   $\alpha_{5,1}$ & = &  0.738333333333333  		& $\gamma_{5,1}$ & = & -0.571666666666667  \\
   $\alpha_{5,2}$ & = & -0.121666666666667	  	& $\gamma_{5,2}$ & = & -0.121666666666667 \\
   $\alpha_{5,3}$ & = &  0.333333333333333  		& $\gamma_{5,3}$ & = &  0.333333333333333	   \\
   $\alpha_{5,4}$ & = &  0.05 					& $\gamma_{5,4}$ & = &  0.05  \\
   $\alpha_{6,1}$ & = & -0.096929102825711		& $\gamma_{6,1}$ & = &  0.263595769492377  \\ 
   $\alpha_{6,2}$ & = & -0.121666666666667		& $\gamma_{6,2}$ & = & -0.121666666666667  \\
   $\alpha_{6,3}$ & = &  1.045582889789120		& $\gamma_{6,3}$ & = & -0.378916223122453  \\
   $\alpha_{6,4}$ & = &  0.173012879703258		& $\gamma_{6,4}$ & = & -0.073012879703258  \\
   $\alpha_{6,5}$ & = &  0.0					& $\gamma_{6,5}$ & = &  0 \\
\hline
   $b_1$ 	 & = &  0.166666666666667		 		& $\widehat{b}_1$	  & = &  0.166666666666667 \\
   $b_2$	 & = & -0.243333333333333 				& $\widehat{b}_2$   & = & -0.243333333333333 \\
   $b_3$	 & = &  0.666666666666667		  		& $\widehat{b}_3$   & = &  0.666666666666667 \\
   $b_4$	 & = &  0.100000000000000 				& $\widehat{b}_4$	  & = &  0.1 \\
   $b_5$	 & = &  0.0			 				& $\widehat{b}_5$   & = &  0.31 		   \\
   $b_6$  & = &  0.31							& $\widehat{b}_6$   & = &  0.0 \\
\hline
  \end{tabular}
\caption{Coefficients of ROK4b, a fourth order, stiffly accurate, type 1 Rosenbrock-K method.}
\label{table:ROK4b-coef}
\end{center}
\end{table}
ROK4b has the additional benefit of both the main and embedded methods are L-stable.  Figure \ref{fig:stability-ROK4} shows the stability functions of the main and embedded methods of {\sc Rok4}b evaluated along the imaginary axis.

%%%%%%%%%%%%%%%%%%%%%%%%
\subsection{ROK4p: a five stages,  fourth order, parabolic Rosenbrock-K method}\label{sec:Rok4p}
%%%%%%%%%%%%%%%%%%%%%%%%

Due to their low stage order Rosenbrock methods can be marred by order reduction when solving initial value problems arising from the semi-discretization of PDEs.  
The following set of additional conditions guarantees the full order of convergence for Rosenbrock methods applied to semi-discrete {\it parabolic} PDEs  \cite{Lubich_1995_linearlyImplicit,Novati_2008_secantROW}:
\begin{equation}
\label{eqn:paracond}
	b^T\, \beta^j\, \left(2\beta^2 \mathds{1} - \alpha^2 \right) = 0 \quad \textrm{for}~~ p-2 \leq j \leq s-1 ~~ \textrm{and} ~~  p \geq 3\,.
\end{equation}
Here $b=(b_i)_{i=1,\dots,s}$,  $\alpha=(\alpha_i)_{i=1,\dots,s}$,, $\beta=(\beta_i)_{i=1,\dots,s}$, $p$ is the order of the method, and $s$ is the number of stages. Multiplications are understood component-wise. We will call a Rosenbrock method {\it parabolic} if it satisfies \eqref{eqn:paracond}.

The order conditions for a five-stage, fourth-order, parabolic Rosenbrock-K are:
\begin{equation}
\label{eqn:ocfives}
\renewcommand{\arraystretch}{1.5}
\begin{array}{clcl}
 (a) & b_1 + b_2 + b_3 + b_4 + b_5 & = & 1 \\
 (b) & b_2 \beta_2' + b_3 \beta_3' + b_4\beta_4'+ b_5\beta_5' & = & p_{2,1}(\gamma) \\
 (c) & b_2\alpha_2^2 + b_3\alpha_3^2 + b_4\alpha_4^2 + b_5\alpha_5^2 & = & \frac{1}{3} \\
 (d) & b_3\beta_{3,2}\beta_2' + b_4(\beta_{4,2}\beta_2' + \beta_{4,3}\beta_3') & \ & \\
     & \indent +b_5(\beta_{5,2}\beta_2' + \beta_{5,3}\beta_3' + \beta_{5,4}\beta_4') & = & p_{3,2}(\gamma) \\
 (e) & b_2\alpha_2^3 + b_3\alpha_3^3 + b_4\alpha_4^3 + b_5\alpha_5^3 & = & \frac{1}{4} \\
 (f) & b_3\alpha_3\alpha_{3,2}\beta_2' + b_4\alpha_4(\alpha_{4,2}\beta_2' + \alpha_{4,3}\beta_3') & \ & \\
     & \indent + b_5\alpha_5(\alpha_{5,2}\beta_2' + \alpha_{5,3}\beta_3' + \alpha_{5,4}\beta_4') & = & p_{4,2}(\gamma) \\
 (g_1) & b_3\alpha_{3,2}\alpha_2^2 + b_4(\alpha_{4,2}\alpha_2^2 + \alpha_{4,3}\alpha_3^2) & \ & \\
     & \indent + b_5(\alpha_{5,2}\alpha_2^2 + \alpha_{5,3}\alpha_3^2 + \alpha_{5,4}\alpha_4^2) & = & \frac{1}{12} \\
 (g_2) & b_3\gamma_{3,2}\alpha_2^2 + b_4(\gamma_{4,2}\alpha_2^2 + \gamma_{4,3}\alpha_3^2) & \ & \\
     & \indent + b_5(\gamma_{5,2}\alpha_2^2 + \gamma_{5,3}\alpha_3^2 + \gamma_{5,4}\alpha_4^2) & = & -\frac{1}{3}\gamma \\
 (h) & b_4 \beta_{4,3}\beta_{3,2}\beta_2' + b_5(\beta_{5,3}\beta_{3,2}\beta_2' + \beta_{5,4}\beta_{4,2}\beta_2' + \beta_{5,4}\beta_{4,3}\beta_3' & = & p_{4,4}(\gamma) \\
 (i) & 2b_5\beta_{5,4}\beta_{4,3}\beta_{3,2}\beta_2' - b_4\beta_{4,3}\beta_{3,2}\alpha_2^2 - b_5\beta_{5,3}\beta_{3,2}\alpha_2^2 & \ & \\
     & \indent - b_5\beta_{5,4}\beta_{4,2}\alpha_2^2 - b_5\beta_{5,4}\beta_{4,3}\alpha_3^2 & = & \pi_1(\gamma) \\
 (j) & b_5\beta_{5,4}\beta_{4,3}\beta_{3,2}\alpha_2^2 & = & \pi_2(\gamma) \\
 (k) & 0 & = & \pi_3(\gamma)
\end{array}
\end{equation}
where the polynomials $p_{i,j}(\gamma)$ are defined  in (\ref{eqn:ocfours}), and 
\begin{eqnarray*}
   \pi_1(\gamma) & = & 2\gamma p_{4,3} - 8\gamma p_{4,4}(\gamma) + \frac{1}{3}\gamma^2 - 12\gamma^2 p_{3,2}(\gamma) - 8\gamma^3p_{2,1}(\gamma) - 2\gamma^4 \,, \\
   \pi_2(\gamma) & = & 3\gamma \pi_1(\gamma) - 3\gamma^2 p_{4,3}(\gamma) + 20\gamma^2 p_{4,4}(\gamma) - \frac{1}{3}\gamma^3 + 20\gamma^3 p_{3,2}(\gamma)  \\
   & &  \indent + 10\gamma^4 p_{2,1}(\gamma) + 2\gamma^5 \,, \\
   \pi_3(\gamma) & = & -4\gamma \pi_2(\gamma) + 6\gamma^2 \pi_1(\gamma) - 4\gamma^3 p_{4,3}(\gamma) + 40\gamma^3p_{4,4}(\gamma) - \frac{1}{3}\gamma^4  \\
   & & \indent  + 30\gamma^4p_{3,2}(\gamma) + 12\gamma^5p_{2,1}(\gamma) + 2\gamma^6\,.
\end{eqnarray*}
The approach to solve the system of equations (\ref{eqn:ocfives}) is similar to that used for (\ref{eqn:ocfours}).  A sequence of linear systems is constructed, and for each system arbitrary choices are made for the values of some parameters.  A numerical genetic optimization algorithm is employed to select free parameter values which lead to method coefficients of acceptable magnitudes.  The coefficients of the resulting method, named ROK4p, are given in Table \ref{table:ROK4p-coef}.

\begin{table}
\begin{center}
  \begin{tabular}{|rcrrcr|}
\hline
\multicolumn{6}{|l|}{   $\gamma$  =  0.572816062482135 } \\
\hline
   $\alpha_{2,1}$ & = & 0.757900000000000 		& $\gamma_{2,1}$ & = & -0.757900000000000 		   \\
   $\alpha_{3,1}$ & = & 0.170400000000000 		& $\gamma_{3,1}$ & = & -0.295086678808293  \\
   $\alpha_{3,2}$ & = & 0.821100000000000			& $\gamma_{3,2}$ & = &  0.178900000000000  \\
   $\alpha_{4,1}$ & = & 1.196218621274069			& $\gamma_{4,1}$ & = & -1.836333117783808  \\
   $\alpha_{4,2}$ & = & 0.297700000000000			& $\gamma_{4,2}$ & = & -0.247700000000000 \\
   $\alpha_{4,3}$ & = &-1.433618621274069  		& $\gamma_{4,3}$ & = &  1.681409044712106 \\
   $\alpha_{5,1}$ & = &-0.010650410785863  		& $\gamma_{5,1}$ & = & -0.197089800872483  \\
   $\alpha_{5,2}$ & = & 0.142100000000000			& $\gamma_{5,2}$ & = & -0.684644029868020 \\
   $\alpha_{5,3}$ & = &-0.129349589214137  		& $\gamma_{5,3}$ & = &  0.166330242942910		   \\
   $\alpha_{5,4}$ & = & 0.392800000000000  		& $\gamma_{5,4}$ & = &  0.000000000000000  \\
\hline
   $b_1$ 	 & = & 0.056000000000000		 		& $\widehat{b}_1$	  & = & -0.186875355621256 \\
   $b_2$	 & = & 0.116601238130482  				& $\widehat{b}_2$   & = & -0.250433793031115 \\
   $b_3$	 & = & 0.160300000000000				& $\widehat{b}_3$   & = &  0.326360736478684 \\
   $b_4$	 & = &-0.031109354304222  				& $\widehat{b}_4$	  & = &  0.110948412173687 \\
   $b_5$	 & = & 0.698208116173739  				& $\widehat{b}_5$   & = &  1.000000000000000		   \\
\hline
  \end{tabular}
\caption{Coefficients of ROK4p, a fourth order, parabolic, type 1 Rosenbrock-K method.}
\label{table:ROK4p-coef}
\end{center}
\end{table}
The choice of $\gamma$ ensures that for the main method  $R(\infty) = 0$.  The embedded method has $\widehat{R}(\infty) = 0.24$. Figure \ref{fig:stability-ROK4} shows the stability function values for both the main and embedded methods along the imaginary axis. We see that the absolute function values are below one, which implies that the main {\sc Rok4}p method is L-stable, and the embedded method is strongly A-stable.
\begin{figure}[htp]
\centering
\includegraphics[width=5.5in,height=2.0in]{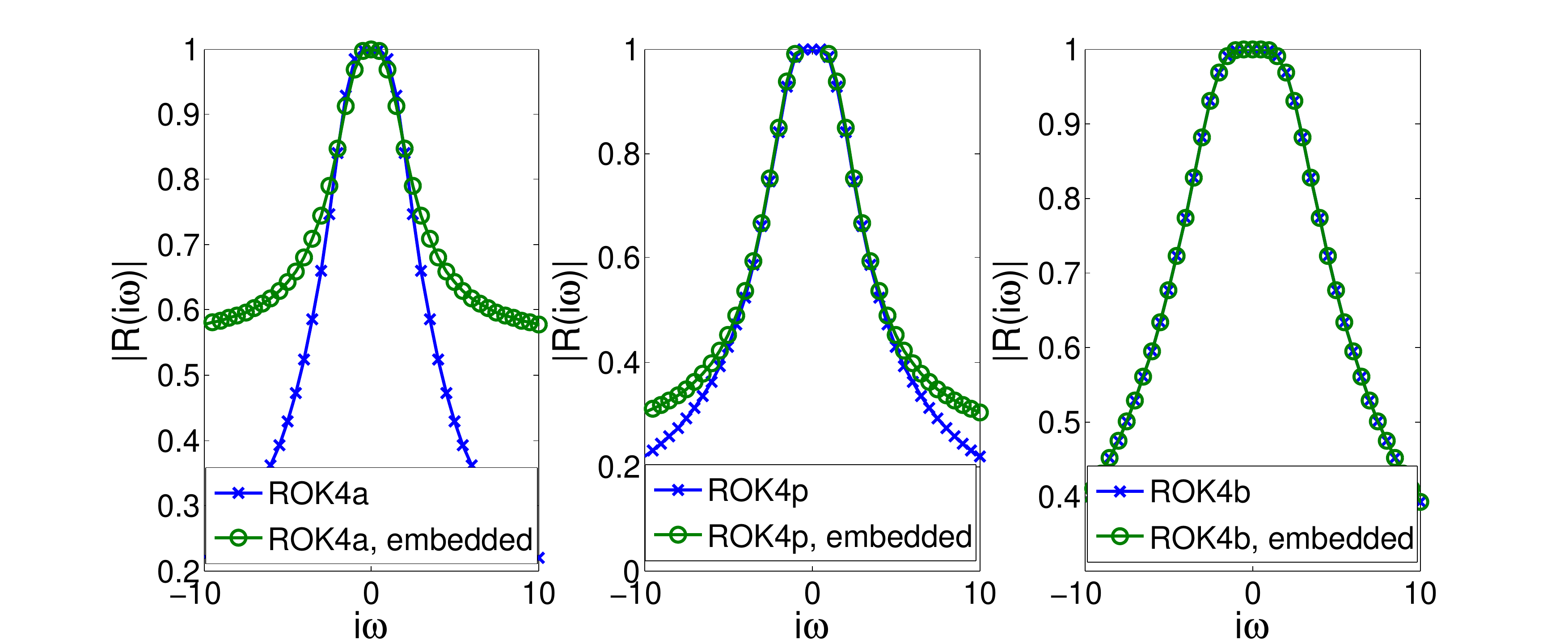}
\caption{Stability functions for the main and embedded methods of {\sc Rok4}a, {\sc Rok4}p, and {\sc Rok4}b.}
\label{fig:stability-ROK4}
\end{figure}
%

%%%%%%%%%%%%%%%%%%%%%%%%
\section{Numerical Results} \label{sec:numericalresults}
%%%%%%%%%%%%%%%%%%%%%%%%

Here we present some results from numerical experiments verifying the properties of the methods discussed above, as well as comparing performance of Rosenbrock-Krylov methods with several standard classical Rosenbrock and Rosenbrock-W methods.  {\sc Rang3} is a third order Rosenbrock-W method \cite{Rang_2005_ROW3}, {\sc Rodas4} is a fourth order, stiffly accurate classical Rosenbrock method \cite[Section IV.10]{Hairer_book_II}, and {\sc Ros4} is a fourth order, L-stable, classical Rosenbrock method \cite[Section IV.10]{Hairer_book_II}.

\subsection{Lorenz 96}
The nonlinear test is carried out with the Lorenz-96 model \cite{Lorenz1996}.  This chaotic model has $N=40$ states, periodic boundary conditions, and is described by the following equations:
\begin{eqnarray}
\label{LorenzModel}
\frac{dy_j}{dt} &=& -y_{j-1}\; \left(y_{j-2}-y_{j+1}\right)-y_j + F \;,
\quad j = 1, \ldots, N~,\\
\nonumber
y_{-1} &=& y_{N-1}~, \quad  y_{0} = y_N ~, \quad  y_{N+1} = y_{1}~.
\end{eqnarray}
The forcing term is $F=8.0$, with $t \in [0, \ 0.3]$.

%%%%%%%%%%%%%%%%%%%%%%%%
\begin{figure}[htp]
\centering
\includegraphics[width=5in]{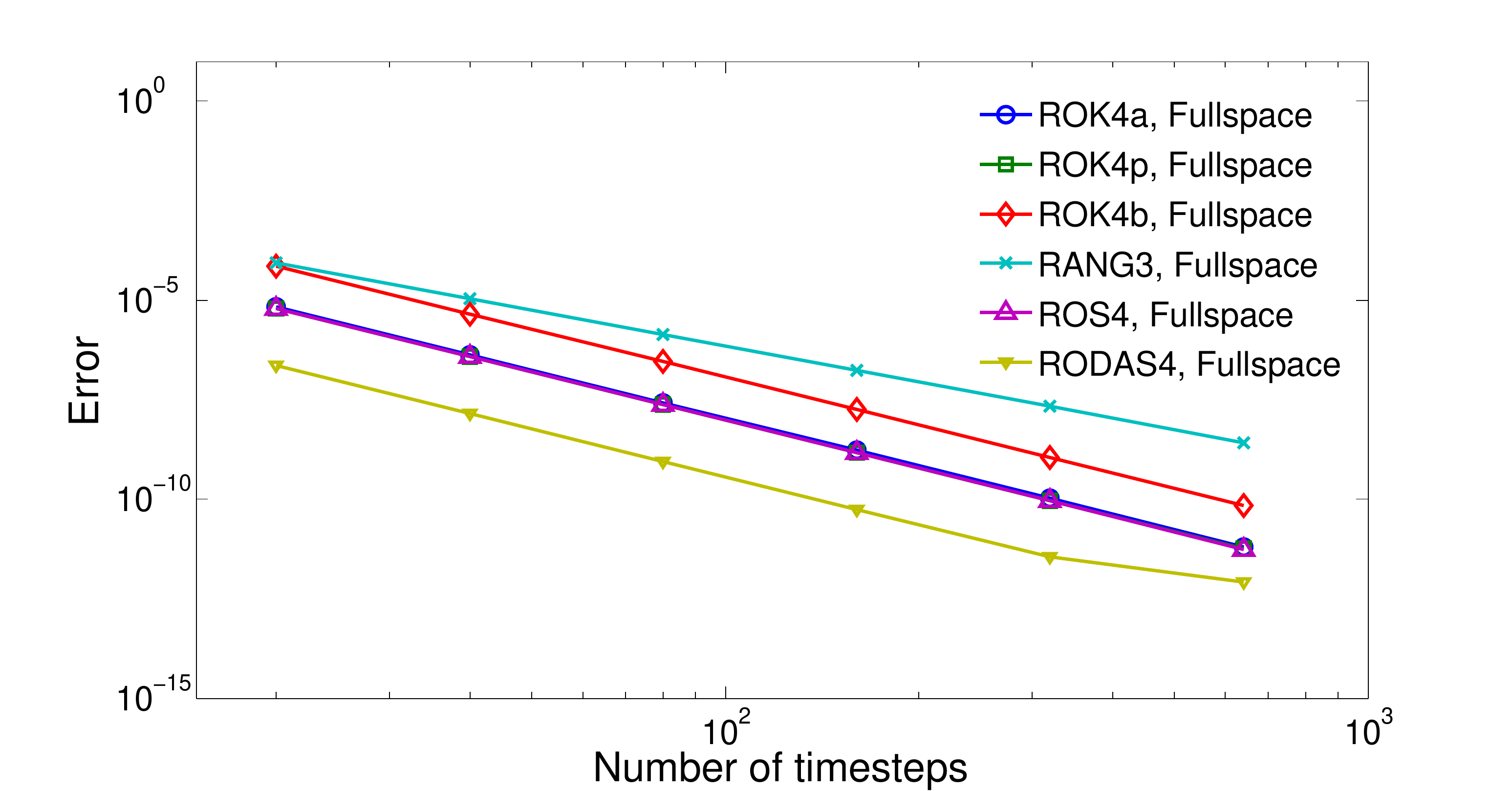} \\
\caption{Precision diagram for Lorenz 96, showing convergence order of methods using a full Jacobian.
}
\label{fig:lorenzconvergence}
\end{figure}
%%%%%%%%%%%%%%%%%%%%%%%%

%%%%%%%%%%%%%%%%%%%%%%%%
\begin{table}[ht]
\centering
\begin{tabular}{|c|c|c|c|c|c|c|}
\hline
		&	{\sc Rang3} 	&	{\sc Ros4}		&	{\sc Rodas4}	& {\sc Rok4}a	& {\sc Rok4}p	  & {\sc Rok4}b \\
\hline
M =  N 	&	2.99		&	4.01		&	3.99		&	4.01		&	3.99		&	3.99 \\
\hline
M = 4	&	2.99		&	3.03		&	3.05		&	4.01		&	3.98		&	3.99 \\
\hline
\end{tabular}
\caption{Convergence Rates on Lorenz 96.}
\label{table:lorenzconvergence}
\end{table}
%%%%%%%%%%%%%%%%%%%%%%%%

Table \ref{table:lorenzconvergence} shows the convergence orders of all methods applied to the Lorenz-96 system, using both the full Jacobian as well as a four dimensional Krylov approximation of the Jacobian.  Figure \ref{fig:lorenzconvergence} verifies numerically the theoretical order results for all methods using the full Jacobian.

Recall that all methods satisfying the classical Rosenbrock order conditions are also Rosenbrock-K methods of at least order three.  Table \ref{table:lorenzconvergence} shows this property, where the third order method {\sc Rang3} maintains its order and both fourth order methods, {\sc Ros4} and {\sc Rodas4}, reduce to third order while using the approximate Jacobian.

%%%%%%%%%%%%%%%%%%%%%%%%
\begin{figure}[htp]
\centering
\includegraphics[width=5in]{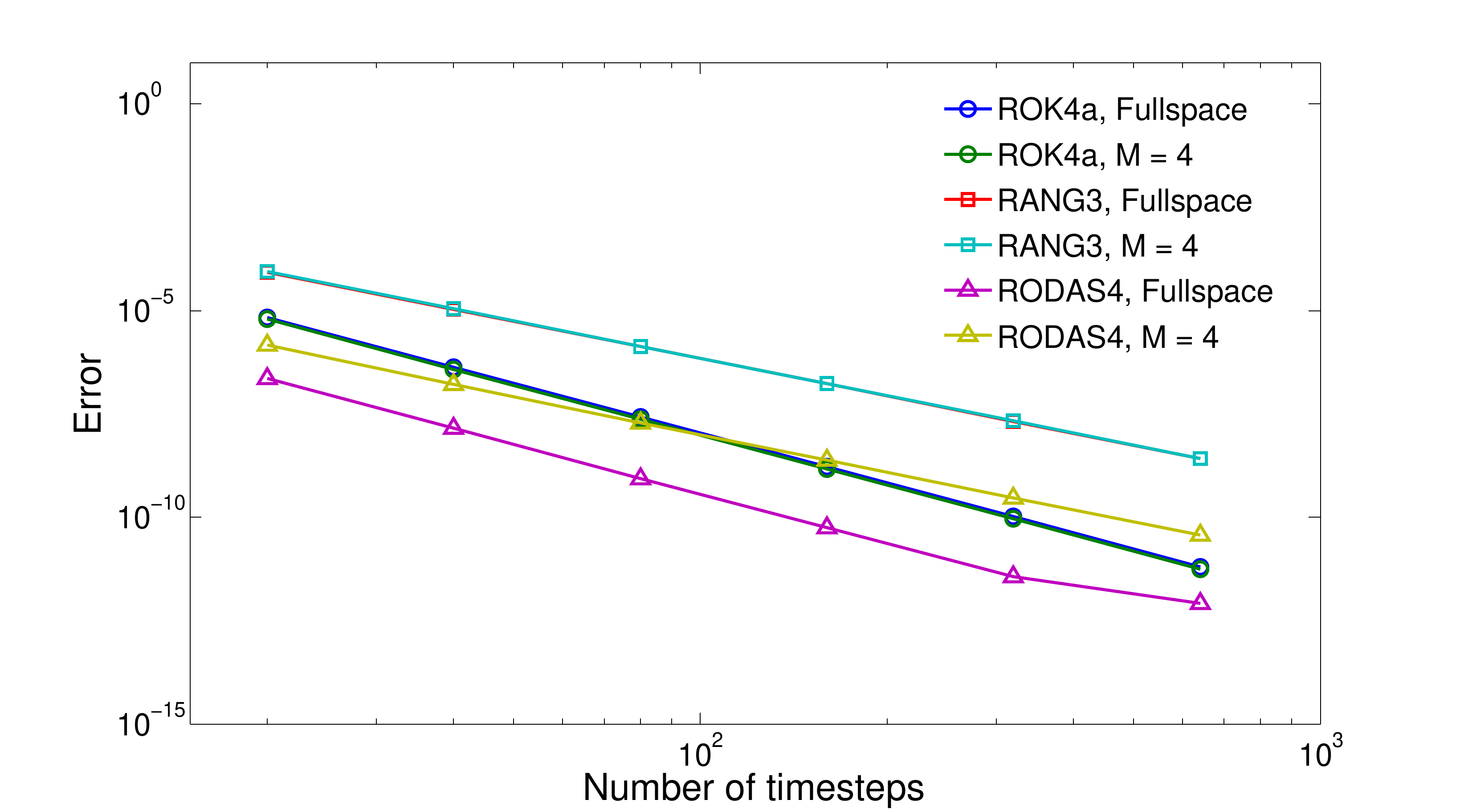}\\
\caption{Precision diagram showing the convergence order of {\sc Rok4}a, {\sc Rang3}, and {\sc Rodas4} using the full and Krylov approximated Jacobian.}
\label{fig:lorenzconvcomp}
\end{figure}
%%%%%%%%%%%%%%%%%%%%%%%%

%%%%%%%%%%%%%%%%%%%%%%%%
\subsection{Dissipative Burger's equation}
%%%%%%%%%%%%%%%%%%%%%%%%
We apply the newly derived methods to an ODE system coming from a semi-discretization of a partial differential equation using the method of lines.  The dissipative Burger's equation is a one-dimensional PDE described by 
%%%%%%%%%%%%%%%%%%%%%%%%
\begin{equation}
	\label{eqn:burger}
	\frac{du}{dt} + \frac{d}{dx}\left(\frac{1}{2}u^2\right) = \varepsilon\frac{d^2u}{dx^2}, \quad  x \in [0,10], \quad  t \in [0, \  0.5], \quad \varepsilon = .001\,,
\end{equation}
%%%%%%%%%%%%%%%%%%%%%%%%
with homogeneous boundary conditions, and initial condition 
\[
u(x,t=0) = \frac{1}{6}\sin^2\, \left(\frac{1}{5}\pi x\right)\, \left(1-x^2\right)\,, \quad \varepsilon = .001\,.
\]
The spatial discretization is a Nodal Discontinuous Galerkin method using equispaced fourth-order elements, making use of the code base provided for \cite{HesthavenWarburton2008}.

Figures \ref{fig:burgerprecision50} and \ref{fig:burgerprecision2000} show a performance comparison of the newly proposed methods with both {\sc Ros4} and {\sc Rodas4}.  The figures show that for problems of even modest size, Rosenbrock-Krylov methods have comparable efficiency with previously existing methods.  The increase in relative efficiency between Rosenbrock-K methods and the classical Rosenbrock methods as the problem size increases is a good indicator that Rosenbrock-K methods are likely to be much more efficient than full space methods as problem size increases, and the benefits of solving a reduced system become more pronounced.

%%%%%%%%%%%%%%%%%%%%%%%%
\begin{figure}[htp]
\centering
\includegraphics[width=5in]{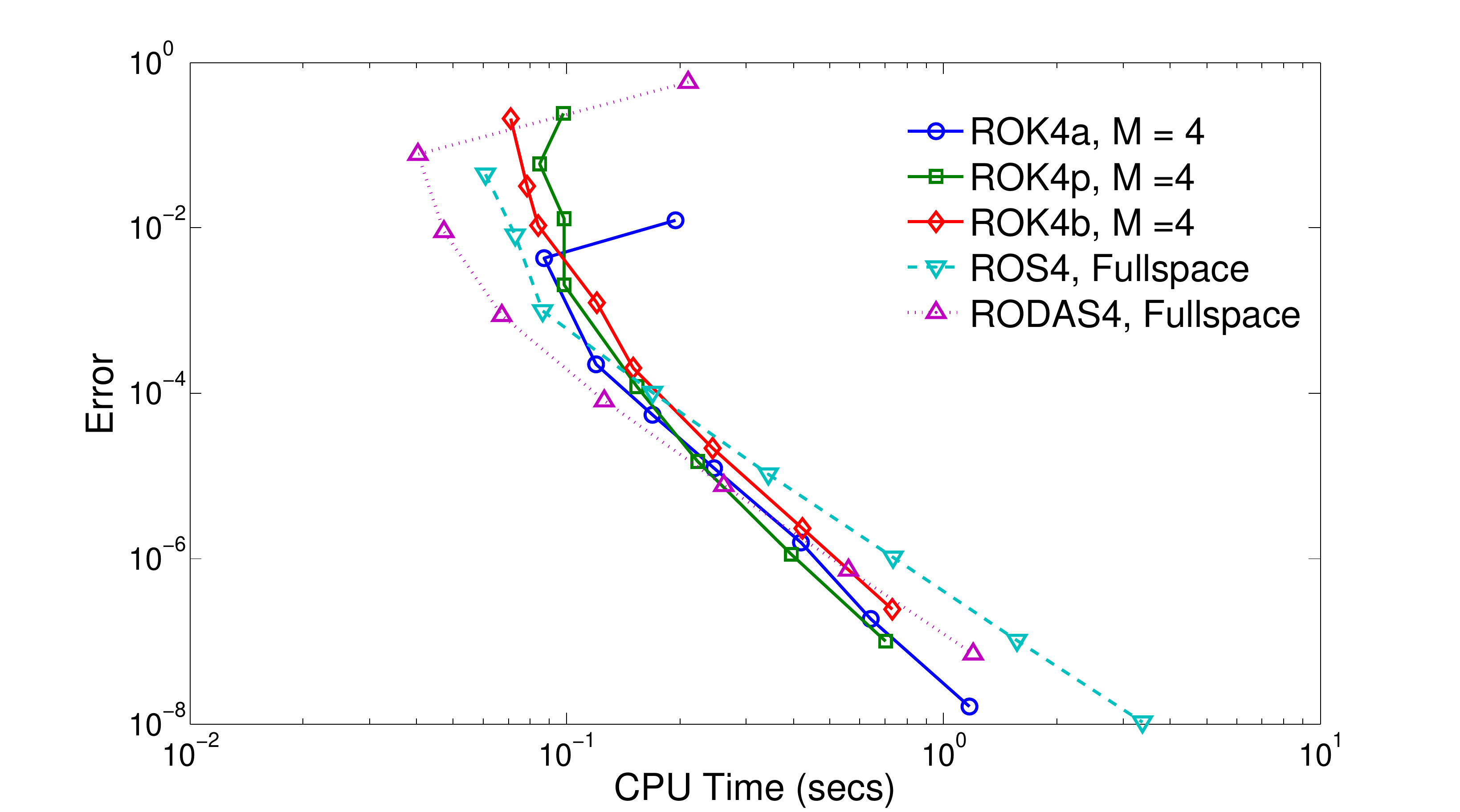}\\
\caption{Precision diagram for Burger's equation using 10 fourth order elements.}
\label{fig:burgerprecision50}
\end{figure}
%%%%%%%%%%%%%%%%%%%%%%%%

%%%%%%%%%%%%%%%%%%%%%%%%
\begin{figure}[htp]
\centering
\includegraphics[width=5in]{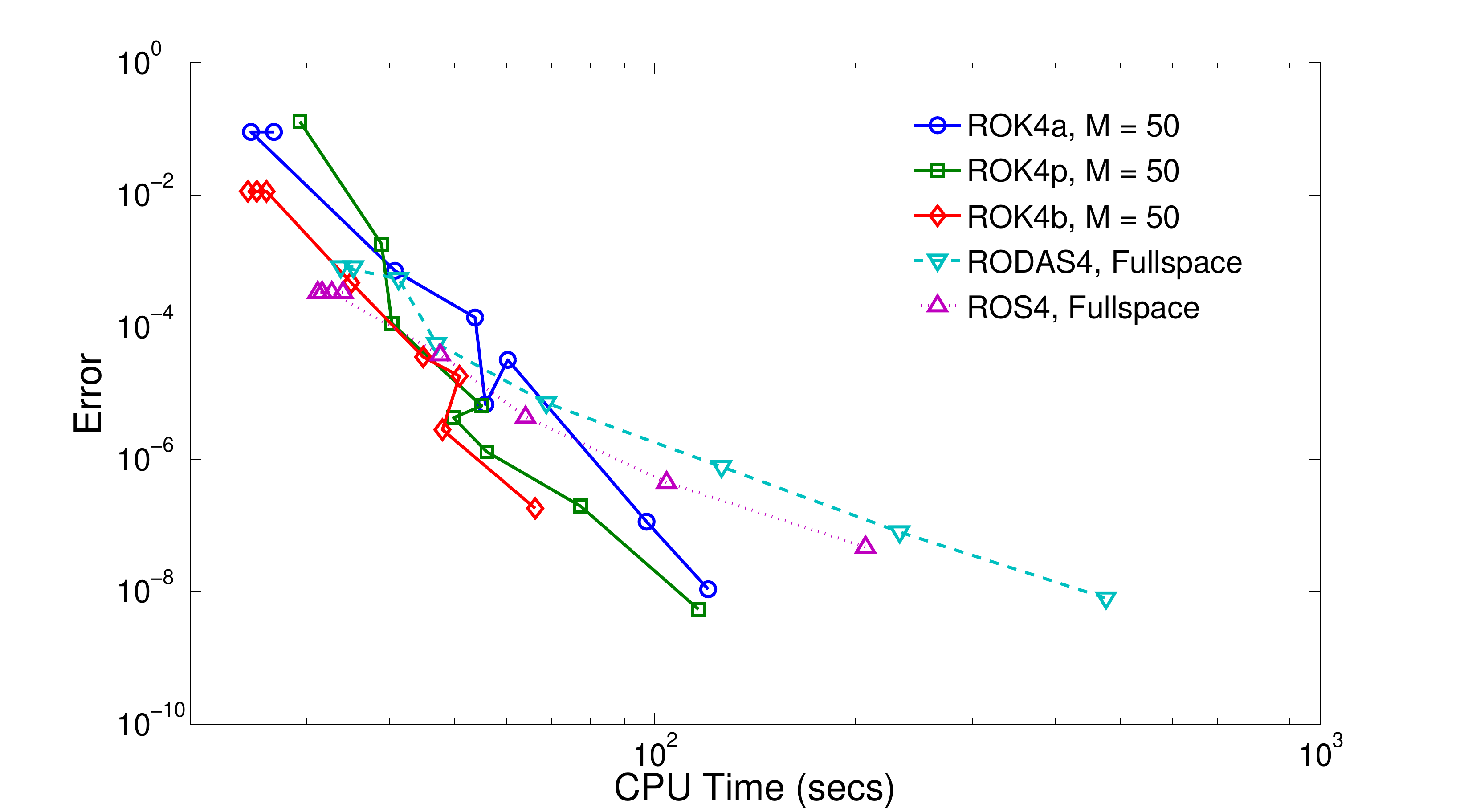}\\
\caption{Precision diagram for Burger's equation using 400 fourth order elements.}
\label{fig:burgerprecision2000}
\end{figure}
%%%%%%%%%%%%%%%%%%%%%%%%

%%%%%%%%%%%%%%%%%%%%%%%%%
%\begin{figure}[htp]
%\centering
%\subfigure[10 fourth order elements]{
%\includegraphics[width=0.45\textwidth,height=0.45\textwidth]{Figures/BurgerAdapt50.pdf}
%}
%\subfigure[400 fourth order elements]{
%\includegraphics[width=0.45\textwidth,height=0.45\textwidth]{Figures/BurgerAdapt2000.pdf}
%}
%\caption{Precision diagram for Burger's equation using different discretizations. \textcolor{red}{This is an example of how things are put together.
%Please scale the text in the figs accordingly, and consistently.} }
%\label{fig:burgerprecision}
%\end{figure}
%%%%%%%%%%%%%%%%%%%%%%%%%

%%%%%%%%%%%%%%%%%%%%%%%%
\subsection{CBM-IV}
%%%%%%%%%%%%%%%%%%%%%%%%

Here we give some results for {\sc ROK} methods applied to a stiff system of ODEs coming from a KPP MATLAB implementation of the CBM-IV model \cite{Gery89photochemicalkinetics}.  This problem is based on the Carbon Bond Mechanism IV (CBM-IV), consisting of 32 chemical species involved in 70 thermal and 11 photolytic reactions \cite{Sandu96benchmarkingstiff}.

%%%%%%%%%%%%%%%%%%%%%%%%
\begin{figure}[htp]
\centering
\includegraphics[width=5in]{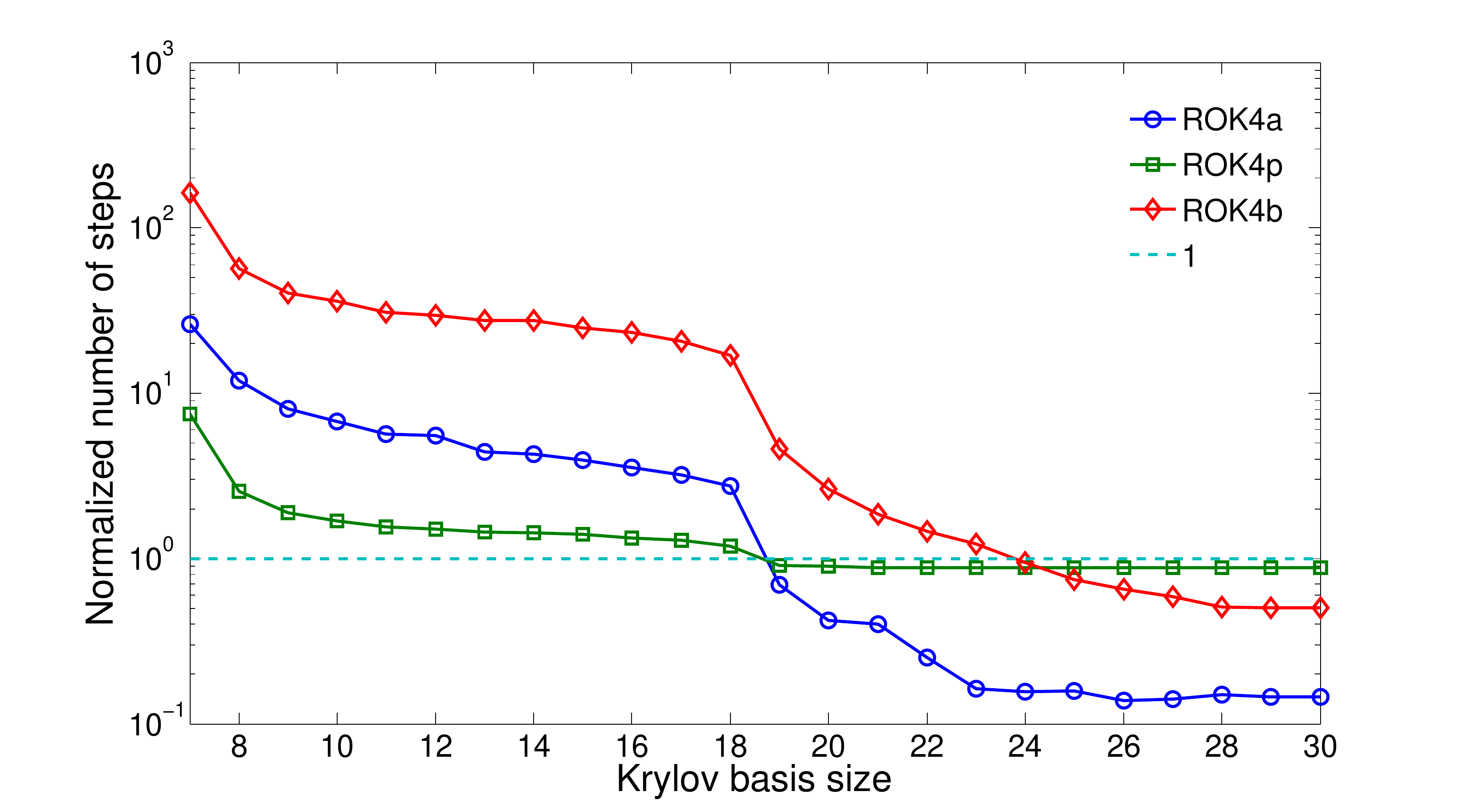}\\
\caption{Number of accepted steps in a full space solution with tolerances of $10^{-2}$ for CBM-IV.}
\label{fig:cbm4}
\end{figure}
%%%%%%%%%%%%%%%%%%%%%%%%

While CBM-IV is a perfect example of a problem for which Rosenbrock-K methods are a poor choice, due to its small size and relatively large number of stiff variables, it does allow us to illustrate numerically the relationship between stability and choice of Krylov basis size. Figure \ref{fig:cbm4} shows the number of timesteps, normalized to a full space solution of the respective method, required to obtain a reasonable solution in a single day simulation of the CBM-IV model.  The number of timesteps required for a full space solution are given in Table \ref{table:cbm4}.

\begin{table}[ht]
\begin{center}
$\begin{array}{|c|c|c|}
\hline
\textrm{{\sc Rok4}a}	& \textrm{{\sc Rok4}p}	  & \textrm{{\sc Rok4}b} \\
\hline
1301 & 10270 & 261 \\
\hline
\end{array}$
\caption{Number of required timesteps in a full space solution with tolerances of $10^{-2}$.}
\label{table:cbm4}
\end{center}
\end{table}

We see from this figure that for small Krylov basis sizes, Rosenbrock-K methods are unstable.  However, as the size of the Krylov space nears the size of the full space the behavior of the Rosenbrock-K methods approaches that of the full space method.

\subsection{Shallow water equations}

We examine the relative performance of the methods on the shallow water equations \cite{Liska97compositeschemes}.
%%%%%%%%%%%%%%%%%%%%%%%%
\begin{subequations}
\label{eqn:shallowwater}
\begin{eqnarray}
 \frac{\partial}{\partial t} h + \frac{\partial}{\partial x} (uh) + \frac{\partial}{\partial y} (vh) &=& 0 \\
 \frac{\partial}{\partial t} (uh) + \frac{\partial}{\partial x} \left(u^2 h + \frac{1}{2} g h^2\right) + \frac{\partial}{\partial y} (u v h) &=& 0  \\
 \frac{\partial}{\partial t} (vh) + \frac{\partial}{\partial x} (u v h) + \frac{\partial}{\partial y} \left(v^2 h + \frac{1}{2} g h^2\right) &=& 0, 
\end{eqnarray}
\end{subequations}
%%%%%%%%%%%%%%%%%%%%%%%%
with reflective boundary conditions, where $u(x,y,t)$, $v(x,y,t)$ are the flow velocity components and $h(x,y,t)$ is the fluid height.
After spatial discretization using centered finite differences on a $32 \times 32$ grid the system \eqref{eqn:shallowwater} is brought to the standard ODE form  (\ref{eqn:ode}) with
%%%%%%%%%%%%%%%%%%%%%%%%
\begin{equation*}
y = \left[ u \, \, v \, \, h\right]^T \in \R^{N}, \quad f_y(t,y) = \mathbf{J} \in \R^{N \times N}, ~~ N = 3072.
\end{equation*}
%%%%%%%%%%%%%%%%%%%%%%%%
For all experiments we report only the time discretization errors, calculated against a reference solution computed by MATLAB's ODE15s solver with 
absolute and relative tolerances set to $10^{-12}$.

Figure \ref{fig:swe} gives an efficiency comparison of the Rosenbrock-K and classical Rosenbrock methods, all methods make use of a sparse Jacobian matrix.  This problem illustrates the scalability of the Rosenbrock-K methods, when the stiff subspace of the problem is kept relatively small.  Here there are $3072$ state variables, but the Rosenbrock-K methods require only eight basis vectors for stability, and so the cost of computing the Krylov space and solving the small system is much smaller than solving the linear system in the full space.

\begin{table}[ht]
\begin{center}
$\begin{array}{|c|c|c|c|}
\hline
 				&	\textrm{{\sc Rok4}a}	& \textrm{{\sc Rok4}p}		& \textrm{{\sc Rok4}b} \\
\hline
\textrm{Fullspace}	&	3.85					& 3.87					& 3.94 \\
\hline
\textrm{Exact Jacobian}	&	3.86			& 3.88					& 3.94 \\
\hline
\textrm{Finite-difference} 	& 3.86			& 3.88					& 3.94 \\
\hline
\textrm{Low accuracy finite-difference} & 1.99		& 0.99					& 2.01 \\
\hline
\end{array}$
\caption{Convergence rates for shallow water equations.}
\label{table:swe-conv}
\end{center}
\end{table}

Table \ref{table:swe-conv} shows the convergence rates of the three {\sc Rok} methods applied with a constant timestep to solve the shallow water problem. 
We consider the exact Jacobian implementation as well as reduced space approximations of the Jacobian. We compute the Krylov space approximations 
using exact Jacobian-vector products and using finite difference approximations. In order to assess the impact of the finite difference errors we consider both an accurate implementation of the Jacobian-vector products, where the increment $\delta$ in \eqref{eqn:jacvec} is carefully chosen, as well as an
inaccurate implementation where a large, fixed value of $\delta$ is used.

%%%%%%%%%%%%%%%%%%%%%%%%
\begin{figure}[htp]
\centering
\includegraphics[width=5in]{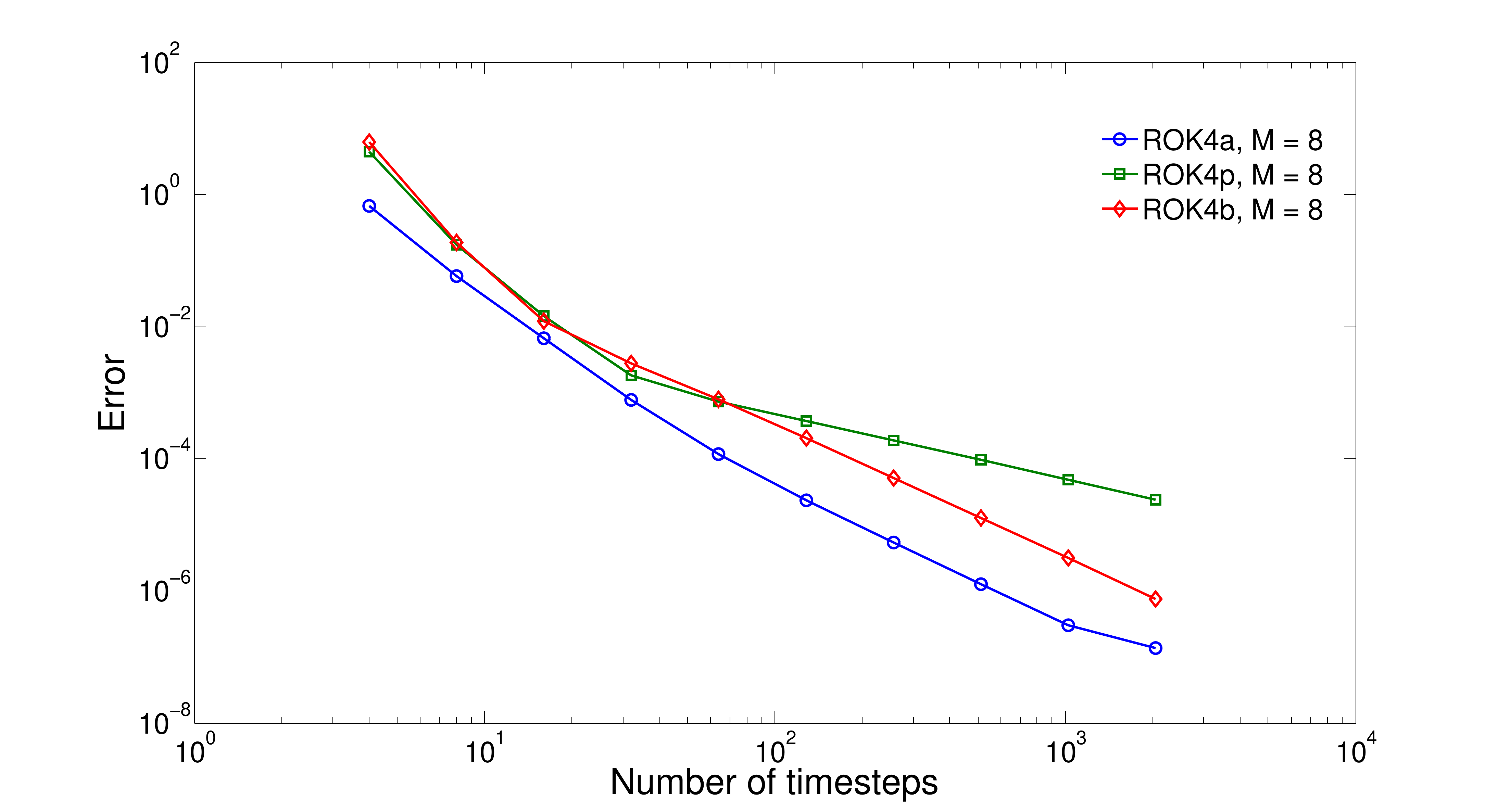}\\
\caption{Precision diagram for different methods on the shallow water equations problem. The lower order convergence rates are due to the low accuracy finite-difference approximation of Jacobian-vector products.}
\label{fig:swebad}
\end{figure}
%%%%%%%%%%%%%%%%%%%%%%%%

To verify the results presented in section \ref{sec:finite-diff} we perform a convergence study for the three {\sc Rok} methods using a low accuracy finite-difference approximation of the Jacobian-vector products applied to the shallow water equations. Results are presented in Figure \ref{fig:swebad}, The three curves show a distinct change in slope.  When the finite-difference error is small compared to the timestep the methods have a convergence orders $3.4$--$3.7$.  When the finite-difference error becomes large compared to the timestep  {\sc Rok4}a and {\sc Rok4}b show second-order and {\sc Rok4}p shows only first-order.  The difference in order when the finite-difference approximation is poor can be explained as follows. {\sc Rok4}a and {\sc Rok4}b satisfy the second-order W condition of Figure \ref{fig:ROWtrees} and have the error term \eqref{eqn:ROW-2}  equal nonzero, while {\sc Rok4}p does not satisfy this additional condition. Rosenbrock-W methods are preferred when finite-differences are the only 
option for obtaining Jacobian-vector products, and these products cannot be obtained accurately.

%%%%%%%%%%%%%%%%%%%%%%%%
\begin{figure}[htp]
\centering
\includegraphics[width=5in]{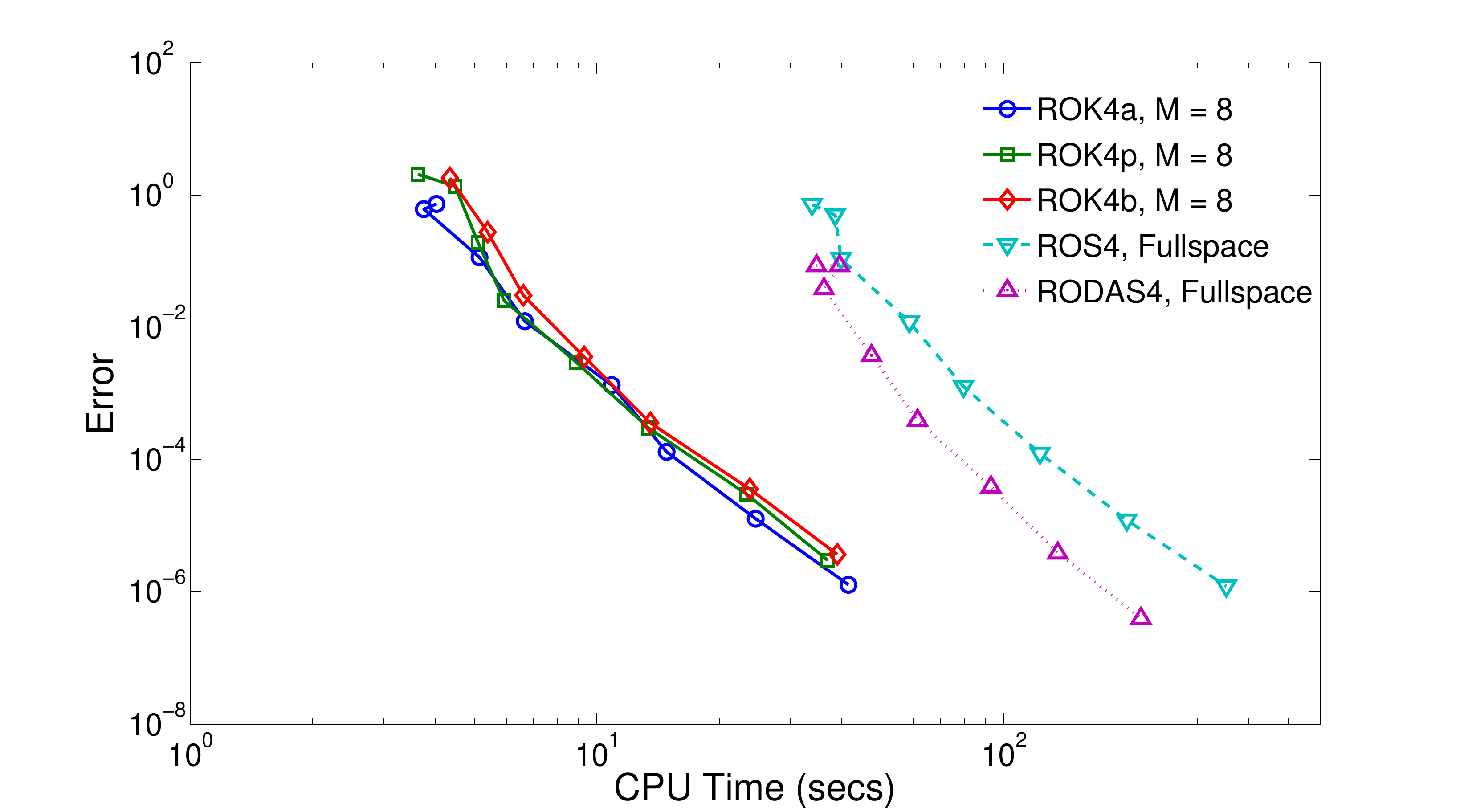}\\
\caption{Precision diagram for the shallow water equations.}
\label{fig:swe}
\end{figure}
%%%%%%%%%%%%%%%%%%%%%%%%

With the small basis size requirements in mind we explore in Figure \ref{fig:sweERK} the relative difference in cost, measured by the number of right hand side evaluations, between an explicit Runge-Kutta method and matrix-free Rosenbrock-K methods.  Figure \ref{fig:sweERK} shows the number of function evaluations on the $x$-axis and the Error of the resulting solution on the $y$-axis.  The number of function evaluations for {\sc Rok}4a includes those required to compute the matrix-free Jacobian-vector products in the Arnoldi iteration.  For the shallow water equations we see that Rosenbrock-K methods perform well against the explicit Runge-Kutta method, when low accuracy is desired and the CFL condition begins to constrain the explicit method.

%%%%%%%%%%%%%%%%%%%%%%%%
\begin{figure}[htp]
\centering
\includegraphics[width=5in]{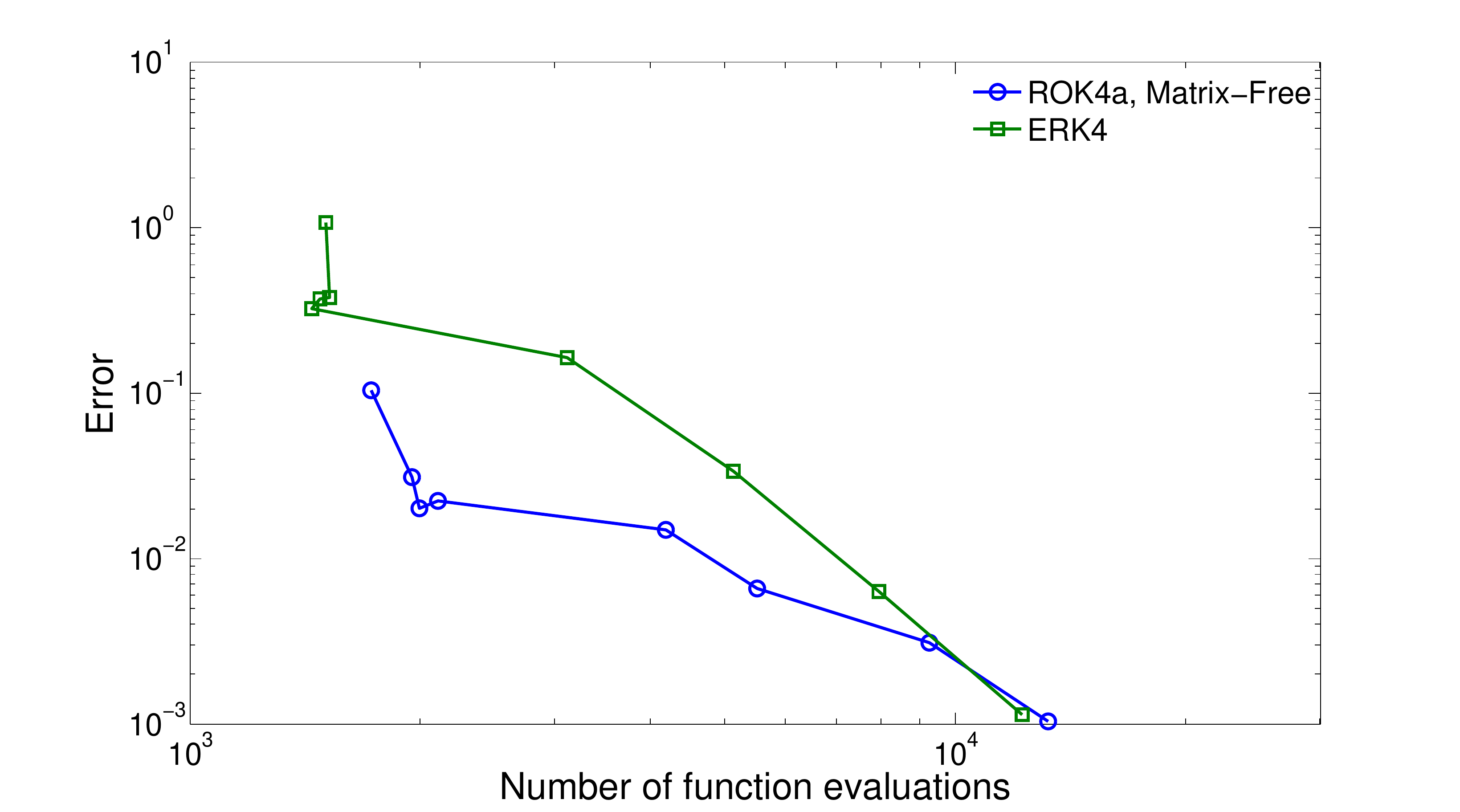}\\
\caption{Comparison of ROKa and ERK4 on the shallow water equations.}
\label{fig:sweERK}
\end{figure}
%%%%%%%%%%%%%%%%%%%%%%%%

%%%%%%%%%%%%%%%%%%%%%%%%
%\begin{table}
%\centering
%$\begin{array}{|c|c|c|c|c|c|}
%\hline
%	\lambda_1 & \lambda_2 & \lambda_3 & \lambda_4 & \lambda_5 & \lambda_6 - \lambda_{32} \\
%\hline
%	-1.4e(-9) & -7.2e(-4) & -3.7e(-3) & -4.22     & -2.27     &  -0.25 < Re(\lambda_i) \leq 0 \\
%\hline
%\end{array}$
%\caption{Eigenvalues of initial Jacobian for CBM-IV}
%\label{table:cbm4eig}
%\end{table}
%%%%%%%%%%%%%%%%%%%%%%%%

%%%%%%%%%%%%%%%%%%%%%%%%
%\begin{table}[htp]
%\centering
%$\begin{array}{|c|c|c|}
%\hline
%  \textrm{} 	& M = 4 	& 	M = 5 	   \\
%\hline
%  \textrm{ROK4}					& 420	&	116	\\
%\hline
%  \textrm{ROK4p}					& 1989	&	1029 \\
%\hline
%  \textrm{ROK4b}					& 415	&	126 \\
%\hline
% \end{array}$
%\caption{Number of steps for convergence on CBM-IV}
%\label{table:cbm4timesteps}
%\end{table}
%%%%%%%%%%%%%%%%%%%%%%%%

%%%%%%%%%%%%%%%%%%%%%%%%
\section{Conclusions and future work} \label{sec:conclusions}
%%%%%%%%%%%%%%%%%%%%%%%%

In this work we have developed a new class of Rosenbrock like integrators, along with a corresponding order condition theory.  We consider the ODE integrator and linear solver as a single computational process to develop methods with the least possible amount of implicitness.

The Rosenbrock-K order conditions remove the requirement for accurate solution of the linear systems which constrain the use of approximate methods in classical Rosenbrock integrators.  For accuracy of the integration process, the size of the Krylov approximation of the Jacobian need be only as large as the desired order of the method.  Stability considerations give stricter requirements on the size of the Krylov basis used, though the exact nature of these requirements is not yet entirely understood and will be the focus of future work.  Some numerical investigation, and a result by Wensch \cite{Wensch_2005_DAE}, give reason to believe that the required size of the Krylov subspace is related to the number of stiff variables in the underlying problem.

The Rosenbrock-K methods developed here have many favorable properties over similar integrators.  Rosenbrock-K methods have substantially fewer order conditions than Rosenbrock-W methods, requiring only a single extra order condition for order four methods as opposed to four extra conditions for order three methods in the case of Rosenbrock-W.  The reduced number of order conditions allows for methods of higher order, we have given conditions up to order five, or for methods with fewer stages.  Further, the structure of Rosenbrock-K methods allows for the computation of a single Krylov subspace at each timestep without the requirement of enriching this space for each internal stage, as is the case for Krylov-ROW methods.

The efficiency of Rosenbrock-K integrators applied to a specific problem is dependent on the stability requirements, and therefore on the stiffness of the underlying problem. We have illustrated this in \cite{Tranquilli13ROK} with the help of  a chemical kinetic test problem. For this reason Rosenbrock-K methods are best suited to very large problems in which there is a relatively small number of stiff variables. However, Rosenbrock-K  methods are expected to perform at least as well as Krylov-ROW methods in all cases.

\section*{Acknowledgements}

This work has been supported in part by NSF through awards NSF
OCI--8670904397, NSF CCF--0916493, NSF DMS--0915047, NSF CMMI--1130667, 
NSF CCF--1218454, AFOSR FA9550--12--1--0293--DEF, AFOSR 12-2640-06,
and by the Computational Science Laboratory at Virginia Tech.

%%%%%%%%%%%%%%%%%%%%%%%%
\bibliographystyle{siam}
\bibliography{Master,ode_krylov,pde_time_implicit,ode_general,ode_multirate,ode_exponential}
%%%%%%%Podhaisky%%%%%%%%%%%%%%%%%
\end{document}